\documentclass[12pt,reqno]{amsart}
\usepackage[headings]{fullpage}
\usepackage{amsmath,amssymb,amsthm}
\usepackage{bbold}
\usepackage[bookmarks=true,%
colorlinks=true,%
linkcolor=blue,%
citecolor=blue,%
filecolor=blue,%
menucolor=blue,%
urlcolor=blue,%
breaklinks=true]{hyperref}
\usepackage{graphicx,subcaption,url}
\usepackage{fullpage,comment}
\usepackage[shortlabels]{enumitem}
\setlist{font=\normalfont}
\usepackage{mathrsfs,mathtools}
\usepackage{tikz,tikz-cd}
\usetikzlibrary{knots,hobby,math,decorations.markings}
\usepackage{ifthen}
\allowdisplaybreaks

\makeatletter
\renewcommand*{\fps@figure}{htpb!}
\makeatother

\newtheorem{theorem}{Theorem}[section]
\theoremstyle{definition}
\newtheorem{proposition}[theorem]{Proposition}
\newtheorem{lemma}[theorem]{Lemma}
\newtheorem{definition}[theorem]{Definition}
\newtheorem{remark}[theorem]{Remark}
\newtheorem{corollary}[theorem]{Corollary}
\newtheorem{conjecture}[theorem]{Conjecture}

\newtheorem{example}[theorem]{Example}

\DeclareMathAlphabet{\AMSbb}{U}{msb}{m}{n}

\providecommand{\abs}[1]{\left|{#1}\right|}
\providecommand{\floor}[1]{\left\lfloor{#1}\right\rfloor}

\DeclareMathOperator{\tr}{tr}

\DeclareMathOperator{\id}{id}

\DeclareMathOperator{\End}{End}

\DeclareMathOperator{\imag}{Im}
\DeclareMathOperator{\Li}{Li}

\def\Im{\imag}

\newcommand{\red}[1]{{\color{red} #1}}
\newcommand{\blue}[1]{{\color{blue} #1}}

\def\BZ{\AMSbb Z}
\def\BQ{\AMSbb Q}
\def\BR{\AMSbb R}
\def\BC{\AMSbb C}

\newcommand{\cx}{\BC}
\newcommand{\ints}{\BZ}

\def\calD{\mathcal D}

\def\calT{\mathcal T}

\def\s{\sigma}

\def\g{\gamma}

\def\th{\theta}

\def\be{\begin{equation}}
\def\ee{\end{equation}}

\def\z{\zeta}

\def\sma#1#2#3#4{\bigl(\smallmatrix#1&#2\\#3&#4\endsmallmatrix\bigr)} % small matrix
\def\vphi{\varphi}

\def\e{\mathbf{e}}  
\def\sma#1#2#3#4{\bigl(\smallmatrix#1&#2\\#3&#4\endsmallmatrix\bigr)}  % small matrices
\def\ga{\gamma}
\def\den{\mathrm{den}}
\def\diag{\mathrm{diag}}

\newcommand{\SL}{\mathrm{SL}}
\newcommand{\PSL}{\mathrm{PSL}}

\newcommand{\mb}{\mathbf}
\newcommand{\fourier}{\mathcal{F}}
\newcommand{\ii}{\mathsf{i}}
\newcommand{\ptorus}{\Sigma_{1,1}}
\newcommand{\qtorus}{\AMSbb{T}}
\newcommand{\ideal}[1]{\langle#1\rangle}
\newcommand{\qteich}{\widehat{\qtorus}}

\newcommand{\balanced}{{\mathrm{bl}}}
\newcommand{\btorus}{\qtorus^\balanced}
\newcommand{\bteich}{\qteich^\balanced}

\newcommand{\marA}{\mathcal{A}}

\makeatletter

\renewcommand\thepart{\@Roman\c@part}%
\renewcommand\part{%
\if@noskipsec \leavevmode \fi
\par
\addvspace{6.7ex}%
\@afterindentfalse
\secdef\@part\@spart}
\def\@part[#1]#2{%
\ifnum \c@secnumdepth >\m@ne
\refstepcounter{part}%
\addcontentsline{toc}{part}{Part~\thepart.\ #1}%
\else
\addcontentsline{toc}{part}{#1}%
\fi
{\parindent \z@ \raggedright
\interlinepenalty \@M
\normalfont
\ifnum \c@secnumdepth >\m@ne
\centering\large\scshape \partname~\thepart.%
\hspace{1ex}%
\fi%
\large\scshape #2%
\markboth{}{}\par}%
\nobreak
\vskip 4.7ex
\@afterheading}
\def\@spart#1{
\refstepcounter{part}%
\addcontentsline{toc}{part}{#1}%
{\parindent \z@ \raggedright
\interlinepenalty \@M
\normalfont
\centering\large\scshape #1\par}%
\nobreak
\vskip 4.7ex
\@afterheading}
\renewcommand*\l@part[2]{%
\ifnum \c@tocdepth >-2\relax
\addpenalty\@secpenalty
\addvspace{0.75em \@plus\p@}%
\begingroup
\parindent \z@ \rightskip \@pnumwidth
\parfillskip -\@pnumwidth
{\leavevmode
\normalsize \bfseries #1\hfil \hb@xt@\@pnumwidth{\hss #2}}\par
\nobreak
\if@compatibility
\global\@nobreaktrue
\everypar{\global\@nobreakfalse\everypar{}}%
\fi
\endgroup
\fi}

\def\l@subsection{\@tocline{2}{0pt}{2pc}{6pc}{}}
\makeatother

\begin{document}
\title{On the Bonahon--Wong--Yang invariants of pseudo-Anosov maps}

\author{Stavros Garoufalidis}
\address{% Max Planck Institute for Mathematics \\
% Vivatsgasse 7, 53111 Bonn, GERMANY \newline
Shenzhen International Center for Mathematics, Department of Mathematics \\
Southern University of Science and Technology \\
1088 Xueyuan Avenue, Shenzhen, Guangdong, China \newline
{\tt \url{http://people.mpim-bonn.mpg.de/stavros}}}
\email{stavros@mpim-bonn.mpg.de}

\author{Tao Yu}
\address{Shenzhen International Center for Mathematics \\
Southern University of Science and Technology \\
1088 Xueyuan Avenue, Shenzhen, Guangdong, China}
\email{yut6@sustech.edu.cn}

\thanks{
  {\em Keywords and phrases}:
  quantum hyperbolic geometry, hyperbolic 3-manifolds, hyperbolic knots,
  cusped hyperbolic 3-manifolds, pseudo-Anosov surface homeomorphisms, 
  Volume Conjecture, Quantum Modularity Conjecture, perturbative Chern--Simons
  theory, Bonahon--Wong--Yang invariants, 1-loop invariants, Baseilhac--Benedetti
  invariants.
}

\date{23 September 2025,First edition 26 December 2024}%{\today}
% \dedicatory{}

\begin{abstract}
  We conjecture (and prove for once-punctured torus bundles) that the
  Bonahon--Wong--Yang invariants of pseudo-Anosov homeomorphisms
  of a punctured surface at roots of unity coincide with the 1-loop invariant
  of their mapping torus at roots of unity. This explains the topological invariance
  of the BWY invariants and how their volume
  conjecture, to all orders, and with exponentially small terms included,
  follows from the quantum modularity conjecture.
  Using the numerical methods of Zagier and the first author, we illustrate how
  to efficiently compute the invariants and their asymptotics to arbitrary order in
  perturbation theory, using as examples the $LR$ and the $LLR$ pseudo-Anosov
  monodromies of the once-punctured torus. Finally, we introduce descendant
  versions of the 1-loop and BWY invariants and conjecture (and numerically check
  for pseudo-Anosov monodromies of $L/R$-length at most 5) that they are related
  by a Fourier transform. 
\end{abstract}

\maketitle

{\footnotesize
\tableofcontents
}

%%%%%%%%%%%%%%%%%%%%%%%%%%%%%%%%%%%%%%%%%%%%%%%%%%%%%%%%%%%%%%%%%%%%%%%%%%%%
%%%%%%%%%%%%%%%%%%%%%%%%%%%%%%%%%%%%%%%%%%%%%%%%%%%%%%%%%%%%%%%%%%%%%%%%%%%%

\section{Introduction}
\label{sec.intro}

In a series of papers~\cite{BWY:I,BWY:II}, Bonahon--Wong--Yang defined invariants
of pseudo-Anosov (in short, pA) homeomorphisms of punctured surfaces at roots of
unity and conjectured that their growth rate is given in terms of the volume of
the hyperbolic mapping torus. It is a folk conjecture that these invariants are
topological 3-manifold invariants, and parts of a 3-dimensional hyperbolic TQFT
at roots of unity, studied years earlier by the pioneering work of
Baseilhac--Benedetti~\cite{BB2}, following initial ideas of Kashaev.
The main feature of these
theories is that they depend on a hyperbolic 3-manifold with nonempty boundary, and
to an $\SL_2(\BC)$-representation of its fundamental group (such as a lift of the
geometric representation), and to a complex root of unity. The invariants themselves
are given by state-sums associated to local pieces, much like the well-known
TQFT of Witten--Reshetikhin--Turaev. Unlike the WRT construction and its axioms though,
the presence of the global $\SL_2(\BC)$-representation makes gluing axioms of the
hyperbolic TQFT involved, disallowing it to be defined for closed 3-manifolds or
to non-hyperbolic 3-manifolds.

On the positive side, hyperbolic TQFT can be thought of as perturbative
complex Chern--Simons theory at the geometric representation and at a fixed root
of unity, and this is the avenue that we will pursue.

As it turns out, perturbative complex Chern--Simons theory at roots of unity leads
to a collection of power series in a variable $h$ for each complex
root of unity and effectively computable from an essential ideal triangulation
of a cusped hyperbolic 3-manifold~\cite{DG1,DG2} and some additional choices.
The topological invariance of this collection of series follows by combining recent
work of~\cite{GSW} and~\cite{GSWZ}, or alternatively older work of Reshetikhin,
Kashaev and others. We will only use the constant terms of the series mentioned above
\begin{equation}
\label{tauM}
\tau_{M,\lambda,m} : \mu'_\BC \to \overline{\BQ}/\mu'_\BC
\end{equation}
which we will call the 1-loop invariants at roots of unity~\cite[Sec.2.2]{DG2}, and
whose detailed definition we give in Section~\ref{sub.1loop} below. Here $M$ is a
cusped hyperbolic 3-manifold, $\lambda$ its canonical longitude, $m\in\ints$ is a
parameter called the descendant index, which is omitted when $m=0$. $\mu'_\BC$
denotes the set of complex roots of unity of odd order, and $\overline{\BQ}$ the
field of algebraic numbers. For a complex root of unity $\z$ of odd order, the 1-loop
invariant $\tau_{M,\lambda,m}(\z) \in \overline{\BQ}$ is defined up to multiplication
by an integer power of $\z^{1/12}$.

On the other hand, 
\begin{equation}
\label{TBWY}
T_{\vphi,m}: \mu'_\BC \to \overline{\BQ}/\mu'_\BC
\end{equation}
denotes the BWY invariant, extended to all complex roots of unity to all order,
without using any absolute values, and using a symmetric definition of the
Fock--Chekhov algebra discussed in Sections~\ref{sub.CF} and~\ref{sub.BWYdef} below. 

Our goal is to explain the following conjecture and its consequences, as well
as to provide a proof for the case of 1-punctured torus bundles. If $\vphi$ is
a surface homeomorphism, we denote by $M_\vphi$ the corresponding mapping torus. As
is well-known, if $\vphi$ is pA then $M_\vphi$ is a hyperbolic
3-manifold~\cite{Thurston}.

\begin{conjecture}
\label{conj.1}
For every pA punctured surface homeomorphism $\vphi$, and every
complex root of unity $\z$ of odd order, we have
\begin{equation}
\label{eq.conj1}
\tau_{M_\vphi, \lambda, m}(\z^2) = \z^{\frac{1}{12}\BZ} \tau_{M_\vphi, \lambda}(1) T_{\vphi,m}(\z) \,.
\end{equation}
\end{conjecture}

Our main theorem is the following.

\begin{theorem}
\label{thm.1}
Conjecture~\ref{conj.1} holds for all pA homeomorphisms of a once-punctured
torus.
\end{theorem}

In fact, in Section~\ref{sub.thm1} we will prove a stronger version of this
theorem, namely both invariants are given by state-sums whose summands syntactically
agree, up to an overall normalization factor!

There are several consequences of the above conjecture.

\begin{itemize}[$\bullet$, leftmargin=0pt, itemindent=*]

\item {\bf Topological invariance.}
The BWY invariant is indeed a topological invariant of a 3-manifold, namely
the mapping torus of the pA homeomorphism.

\item {\bf Effective computation.}
The BWY invariant, which takes values in the field of algebraic
numbers, is effectively computable both exactly and numerically to any desired order
of precision. In fact, the invariant for a pA map $\vphi$ of a once-punctured
torus with $L/R$-length $N$ at a root of unity of order $n$ has time complexity
$O(N n^3)$ and space complexity $O(n)$; see Section~\ref{sub.compute} below. 

\item {\bf Asymptotics.}
The above conjecture, together with the quantum modularity
conjecture, implies the volume conjecture of the BWY and the 1-loop invariants to all
orders and with exponentially small terms included. In fact, the asymptotic expansion
of the said invariants can be effectively computed using the numerical methods
of~\cite{GZ:kashaev}. We will illustrate those methods in Section~\ref{sec.asy}
with two examples of pA maps of the once-punctured torus, namely the 
standard choice of $LR$ (which corresponds to the simplest hyperbolic $4_1$ knot)
and the case of $LLR$ which exhibits further phenomena not seen by the highly
symmetric $LR$. To whet the appetite, the BWY invariant of the $LR$ given
in Equation~\eqref{BWYLR}, satisfies
\begin{equation}
T_{LR}(e^{2\pi \ii/20001}) \approx 4.0108263579\times 10^{1402}
\end{equation}
and 
\begin{equation}
\label{LRasyfew}
T_{LR}(e^{2\pi\ii/n}) \sim \frac{1}{\sqrt{2}}
\big(1 - \frac{(-1)^{(n-1)/2}}{\sqrt{3}} \big)
e^{\frac{v}{2}(n -1/n)} \hat\Phi_{LR} \Big(\frac{4 \pi \ii}{3 \sqrt{-3} n} \Big)
\end{equation}
for odd $n \to \infty$, where
\begin{equation}
\hat\Phi_{LR}(\hbar) = 1 + \frac{17}{24} \hbar + \frac{2305}{1152} \hbar^2
+ \frac{4494181}{414720} \hbar^3 + \frac{3330710213}{39813120} \hbar^4 
+ \frac{5712350244311}{6688604160} \hbar^5 + \dotsb
\end{equation}
and
\begin{equation}
v_{LR}=\frac{\ii \mathrm{Vol}_{LR}}{2 \pi \ii} \approx 0.323, \qquad
\mathrm{Vol}_{LR} = 2 \Im \Li_2(e^{2 \pi \ii/6}) \,.
\end{equation}

%% see pari/BWY_extended.gp

\begin{comment}
, namely
functions of the form
\begin{equation}
\label{2desc}
\tau_{M,\lambda,m}:   \mu'_\BC \to \overline{\BQ}/\mu'_\BC,
\qquad T_{\vphi,m} : \mu'_\BC \to \overline{\BQ}/\mu'_\BC,
\qquad m \in \BZ.
\end{equation}

The first one is a natural
extension of Conjecture~\ref{conj.1} which we prove for pA maps of the
once-punctured torus.

\begin{conjecture}
\label{conj.2}
For every pA homeomorphism $\vphi$ of a punctured surface, and every
complex root of unity $\z$ of odd order, and every integer $m$ we have
\begin{equation}
\label{eq.conj2}
\tau_{M_\vphi, \lambda,m}(\z^2) = \z^{\frac{1}{12}\BZ}
\tau_{M_\vphi, \lambda}(1) T_{\vphi,m}(\z) \,.
\end{equation}
\end{conjecture}
\end{comment}

\item {\bf Descendants.} A final consequence is the descendant
families of the 1-loop and of the BWY invariants at roots of unity. There are two notable features of these functions.

The first feature is that when $\z$ is a root of unity of order $n$, the descendants
are $n$-periodic functions of $m$, which leads to the following Fourier transform
conjecture relating the 1-loop invariants with respect to the longitude
$\tau_{M,\lambda,m}$ consider in this paper to the 1-loop invariants with respect to
the meridian $\tau_{M,\mu,m}$ considered in~\cite{DG2,GZ:kashaev}.

\begin{conjecture}
\label{conj.3}
Fix a cusped hyperbolic 3-manifold $M$. There is a choice of meridian $\mu$ such that
for all roots of unity $\z$ of odd order $n$ and all integers $m$ we have
\begin{equation}
\label{eq.ff1}
\frac{1}{\sqrt{n}} \sum_{\ell \bmod n} \z^{m \ell}
\frac{\tau_{M,\lambda,\ell}(\z)}{\tau_{M,\lambda}(1)}
= \frac{\tau_{M,\mu,m}(\z)}{\tau_{M,\mu}(1)}
\end{equation}
up to a $12n$-th root of unity.
\end{conjecture}
Equivalently for $M=M_\vphi$, Conjecture~\ref{conj.1} and~\eqref{eq.ff1} imply that
\begin{equation}
\label{eq.ff2}
\frac{1}{\sqrt{n}} \sum_{\ell \bmod n} \z^{2 m \ell}
T_{\vphi,\ell}(\z)
= \frac{\tau_{M_\vphi,\mu,m}(\z^2)}{\tau_{M_\vphi,\mu}(1)} \,.
\end{equation}

The second feature of the descendant invariants is that they are $q$-holonomic
functions of $m$. We illustrate this explicitly in Section~\ref{sub.qholo} for the
$4_1$ knot, and use it to draw conclusions about the asymptotic expansions of
the descendant invariants when $\z=e^{2 \pi \ii /n}$ with odd $n \to \infty$.

\end{itemize}

%%%%%%%%%%%%%%%%%%%%%%%%%%%%%%%%%%%%%%%%%%%%%%%%%%%%%%%%%%%%%%%%%%%%%%%%%%%%
%%%%%%%%%%%%%%%%%%%%%%%%%%%%%%%%%%%%%%%%%%%%%%%%%%%%%%%%%%%%%%%%%%%%%%%%%%%%

\section{Invariants}
\label{sec.invariants}

In this section we review the two key players of the paper, namely the 1-loop
invariants of a cusped hyperbolic 3-manifold and the BWY invariants of a pA
homeomorphism of a punctured surface.

\subsection{A review of the 1-loop invariant at roots of unity}
\label{sub.1loop}

The 1-loop invariants of a cusped hyperbolic 3-manifold at a complex root of unity
are the constant terms of power series expansions at roots of unity with very
interesting arithmetic properties explained in detail in~\cite{GSWZ}. The power series
are defined using as input an essential ideal triangulation of a cusped hyperbolic
3-manifold and a complex root of unity $\z$. These series are essentially the
perturbative expansion of complex Chern--Simons theory at the geometric representation
introduced in~\cite{DG1} when $\z=1$ and in~\cite{DG2} for general $\z$. 
The topological invariance of these series was shown in~\cite{GSW} when $\z=1$.
For our purposes, we will only need the constant terms of the above-mentioned
power series at roots of unity, which are none other than the 1-loop invariants
of~\cite{DG2}. The topological invariance of the latter are discussed in detail
in~\cite{GW:1loop}. 

We now review the definition of the 1-loop invariants of~\cite[Defn.2.1]{DG2} at
roots of unity. The definition is explicit and computer-implemented both numerically
and exactly.

The invariants depend on some combinatorial data on an ideal triangulation that
we now discuss. 
We fix an oriented hyperbolic manifold $M$ with one cusp (for instance
a hyperbolic knot complement) and an oriented ideal triangulation $\calT$ of $M$ 
containing $N$ tetrahedra $\Delta_j$ for $j=1,\dots,N$.

A choice of quad of an oriented tetrahedron is a choice of a pair of opposite
edges. Given such a choice and the orientation of a tetrahedron, we can attach
variables $z$, $z'=1/(1-z)$ and $z''=1-1/z$ at the edges as shown in
Figure~\ref{fig-tetra}. These variables, often called shapes, satisfy the
relations
\begin{equation}
z z' z'' = -1, \qquad z^{-1} + z'' = 1, \qquad (z')^{-1} + z = 1,
\qquad (z'')^{-1} + z' = 1 \,.
\end{equation}

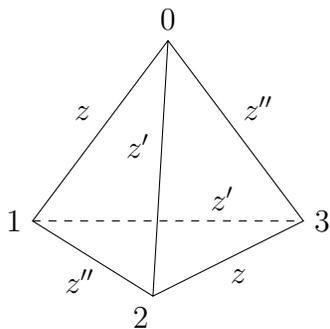
\begin{figure}
\centering
\begin{tikzpicture}[baseline=0.5cm] %0.5cm
\path (0,2)coordinate(A) (-1.8,-0.4)coordinate(B) (-0.2,-1.4)
coordinate(C) (1.8,-0.4)coordinate(D);
\draw (A) 
  -- node[midway, anchor=-45]{$z$} (B)
  -- node[midway, anchor=60]{$z''$} (C)
  -- node[midway, anchor=120]{$z$} (D)
  -- node[midway, anchor=-135]{$z''$} cycle
  (A) -- node[midway, anchor=-45]{$z'$} (C);
\draw[dashed] (B) -- node[pos=0.7, above]{$z'$} (D);
\path (A)node[above]{0} (B)node[left]{1} (C)node[anchor=60]{2} (D)node[right]{3};
\end{tikzpicture}
\caption{Labeling a tetrahedron.}\label{fig-tetra}
\end{figure}

The choice of quad, combined with the orientation of $\calT$ and $M$ allow us
to attach variables $(z_j,z_j',z_j'')$ to each tetrahedron $\Delta_j$.
An Euler characteristic argument shows that the triangulation has 
$N$ edges $e_i$ for $i=1,\dots,N$. Fix peripheral curves
$\mu$ and $\lambda$ that form a symplectic basis for $H_1(\partial M,\BZ)$.

The gluing equation matrices $G$, $G'$ and $G''$ of $\calT$ are $(N+2)\times N$
matrices with integer entries whose columns are indexed by the tetrahedra
$\Delta_j$ of $\calT$ and whose rows are indexed by the edges $e_i$ of $\calT$
for $i=1,\dots,N$ followed by the two peripheral curves $\mu$ and $\lambda$. These
matrices record the number of times each tetrahedron winds around an edge, or a
peripheral curve. Explicitly, the $(i,j)$-entry of $G^\square$ for
$\square \in \{ \ , ' , ''\}$ is the number of $z_j^\square$-labeled edges of
$\Delta_j$ go around an edge $e_i$ of $\calT$; and similarly for the two
peripheral curves. 

The rows of these matrices determine the gluing equations of $\calT$ given by
\begin{equation}
\label{ge}
\sum_{j=1}^N \big(\mb{G}_{ij} \log z_j+\mb{G}_{ij}' \log z_j'
+\mb{G}_{ij}'' \log z_j''\big) = \pi \ii \boldsymbol{\eta}_i, \qquad i=1,...,N+2\,,
\end{equation}
where $\eta=(2,\dots,2,0,0)^t \in \BZ^{N+2}$. 

If $\calT$ is essential, there is a distinguished solution to the gluing equations,
together with the Lagrangian equations
\begin{equation}
\label{lag}
\log z_j + \log z_j' + \log z_j'' = \pi \ii, \qquad j=1,\dots,N
\end{equation}
at each tetrahedron that recovers the completely hyperbolic structure on $M$. 

The gluing and Lagrangian equations can be reduced in two steps as follows.
First, we can eliminate one of the variables $z_j$, $z_j'$ and $z_j''$ (say
$z_j''$) using the Lagrangian equations to obtain the equations 
\begin{equation}
\label{ge2}
\sum_{j=1}^N \big(\mb{A}_{ij} \log z_j' +\mb{B}_{ij} \log z_j\big)
= 2\pi \ii \boldsymbol{\nu}_i, \qquad i=1,...,N+2 
\end{equation}
where
\begin{equation}
\mb{A} = \mb{G}'-\mb{G}'', \qquad \mb{B} = \mb{G}-\mb{G}'', \qquad 
\boldsymbol{\nu}=\boldsymbol{\eta}-\mb{G}_{ij}''(1,\dots,1)^t \,.
\end{equation}
Second, one of the edge gluing equations is redundant, since by the combinatorics of
the triangulation, the sum of the first $N$ rows of $G^\square$ is $(2,\dots,2)$.
So, we can remove one edge-row of $(\mb{A} |\mb{B})$ and keep only one
row of a peripheral curve $\gamma$ resulting to three $N \times N$ matrices $A$ and
$B$ and a vector $\nu \in \BZ^N$ (or better, $A_\ga$, $B_\ga$ and $\nu_\ga$ to
emphasize their dependence on the peripheral curve chosen). 

The last ingredient that we need is a flattening, that is two vectors 
$f,f'\in \BZ^N$ satisfying
\begin{equation}
\label{flat}
A f' + B f = \nu\,.
\end{equation}
The vectors $f$, $f'$ and $f''=1-f-f'$ also label the edges of tetrahedra,
and satisfy with the property that the sum around any edge of the 
triangulation is $2$.

Altogether, the tuple $\Gamma=(A,B,\nu,z,f,f')$ where $z$ is the distinguished
solution of the gluing and Lagrangian equations was called a Neumann--Zagier datum
of the ideal triangulation $\calT$ in~\cite{DG1}. 
We stress that a Neumann--Zagier datum depends not just on the triangulation $\calT$,
the choice of the removed edge, and the included cusp equation, 
but also on the choice of which edges of each tetrahedron are labelled by the
distinguished shape parameter $z_i$; this $3^N$-fold choice has been called
a choice of ``quad'' or ``gauge''.

An important property of the matrix $(A|B)$ is that it is the upper half of
a symplectic matrix over the integers, as shown by Neumann--Zagier for cusped
hyperbolic manifolds in~\cite{NZ} and by Neumann for all 3-manifolds with torus
boundary components~\cite{Neumann}. It follows that $AB^t$ is symmetric
and that $(A|B)$ has full rank $N$. Thus, if $B$ is invertible, $B^{-1}A$ is
symmetric.
%We will call a Neumann-Zagier datum $\BZ$-\emph{nondegerate} if $B$ is 
%invertible over the integers.

The definition of the 1-loop invariant at roots of unity uses 
a primitive complex root of unity $\z$ of order $n$, a $\BZ$-nondegenerate
NZ datum $\Gamma$, and choice $\theta_j$ so that $\theta_j^n=z'_j$ for $j=1,\dots,N$.

It also uses two special functions, the quantum Pochhammer symbol
\begin{equation}
\label{qp}
(x;q)_k = (1-x)(1-qx) \dots (1-q^{k-1}x)
\end{equation}
and the cyclic quantum dilogarithm
\begin{equation}
\label{Ddef}
D_\z(x)=\prod_{j=1}^{n-1}(1-\zeta^{j}x)^j 
\end{equation}
of Kashaev--Mangazeev--Stroganov~\cite[Eqn.C.3]{Kashaev:star} which curiously
predated the definition of the Kashaev invariant~\cite{K95}.
%% and the definition of the Faddeev quantum dilogarithm?

When $\zeta=e^{2 \pi \ii a/n}$ with $(a,n)=1$, the definition of the
invariant requires an $n$-th root of $D_\z(x)$ with a correction, defined by
\begin{equation}
\label{CalDdef}
\calD_\z(x)=\exp \Big({-\ii\pi}s(a,n)
+ \sum_{j=1}^{n-1} \frac{j}{n}\log(1-\z^j x) \Big) \,,
\end{equation}
where $s(a,n)$ is the Dedekind sum; see e.g.,~\cite{Rademacher}. The addition of the
Dedekind sum is chosen so that $\calD_\z(1)=\sqrt{n}$. This correction also appears
in the computations of numerical asymptotics of the Kashaev invariant of the $5_2$
knot; see~\cite[Eqn.(7.12)]{GZ:kashaev}.

Given a vector $v$, we denote by $\diag(v)$ the corresponding diagonal matrix. 

\begin{definition}
\label{def.1loop}
Fix an NZ datum $\Gamma$ with $\tfrac{1}{d} B$ unimodular for some positive integer
$d=1,2$. The $m$-th descendant 1-loop invariant of $\Gamma$ at roots of unity is the
function $\tau_{\Gamma,m}: \mu'_\BC \to \overline{\BQ}/\mu'_\BC$
%(where $\mu_\BC[1/d]$ denotes the set of complex roots of unity of order prime to $d$)
given by 
\begin{equation}
\label{eq.tau}
\frac{\tau_{\Gamma,m}(\z)}{\tau_\Gamma(1)}=
\frac{1}{n^{N/2} z'^{\frac{1-n}{2n}f} z^{\frac{n-1}{2n}f'}
%  \sqrt{\det( A \,\diag(z)
%  + B\diag(z'^{-1})) z'^{f/n}z^{-f'/n}}
} 
\prod_{i=1}^N \calD_{\z^{-1}}(\th_i^{-1}) \sum_{k\in (\BZ/n\BZ)^N} a_{k,m}(\th) 
\end{equation}
where $n$ is the order of $\zeta$, and for $k=(k_1,\dots,k_N) \in (\BZ/n\BZ)^N$,
\begin{equation}
\label{amth}
a_{k,m}(\theta) = (-1)^{d k^t B^{-1} \nu} 
\zeta^{\frac12 \big[d^2 k^t  B^{-1} A k+ d k^t  B^{-1}(\nu-2me_N)\big]}
\prod_{i=1}^N
\frac{\th_i^{-( d B^{-1}A k)_i}}{(\zeta \th_i^{-1};\zeta )_{d k_i}},
\end{equation}
%Here, we define $\z^{\frac{1}{2}}:=\z^{\frac{n+1}{2}}$ for odd $n$. 
and 
\begin{equation}
\tau_\Gamma(1) = \frac{1}{\sqrt{\det( A\diag(z) + B\diag(z'^{-1})) z'^{f}z^{-f'}}} \,.
\end{equation}
Here, $\frac{1}{2}$ is interpreted as $2^{-1}\bmod{n}$, and $e_N\in\ints^N$ is the
unit vector in the $N$-th direction. We mostly consider the case $m=0$, in which case
we omit it from the notations.
\end{definition}

The order of the root of unity is the level of the complex Chern--Simons
theory in~\cite{DG2}. The above definition differs from the one in~\cite{DG2} by
a cyclic rotation of the shapes, but the invariant does not change under
such a rotation (i.e., under a change of quad). We have chosen the above choice of
quad to make the 1-loop invariant syntactically match with the BWY invariant of
once-punctured tori. Note that the quantity inside the square root of $\tau_\Gamma(1)$
is conjectured to equal to the
adjoint Reidemeister torsion~\cite{DG1}. The latter requires a choice of a
peripheral element at each boundary component, due to the non-acyclicity of the
chain complex that defines that torsion~\cite{Porti}. This choice of peripheral
curve which is necessary when $\z=1$ carries to the 1-loop invariant at general
roots of unity.

If $M$ is a cusped hyperbolic manifold that has a canonical meridian $\mu$ (such as
in the case of a hyperbolic knot complement or a hyperbolic
mapping torus), we will denote the corresponding invariant by $\tau_{M,\mu,m}(\z)$.
Likewise, we will denote by $\tau_{M,\lambda,m}(\z)$ the 1-loop invariant with respect
to the longitude (the latter always exists), with the convention that we 
will \emph{halve} its gluing equation, as was done in~\cite[Eqn.(4.6)]{DG1} in
accordance with the fact that the eigenvalue of the longitude at the geometric
representation is always $-1$. 

\begin{remark}
\label{rem.useful}
There is some freedom in the formula for the 1-loop invariant at roots of unity,
which can be achieved using the useful formulas:
\begin{align}
\label{i1}  
(x;q^{-1})_n & = \frac{1}{(qx;q)_{-n}} \\
\label{i2} 
(x;q)_{n+m} & = (x;q)_n (q^n x;q)_m \\
\label{i3}
(x;q)_n & = (-1)^n x^n q^{n(n-1)/2} (x^{-1};q^{-1})_n 
\end{align}
We also use the notation
\begin{equation}
\label{edef}
\e(x)=e^{2 \pi \ii x}, \qquad x \in \BQ \,.
\end{equation}
\end{remark}

\subsection{The 1-loop invariant of the $4_1$ knot}
\label{sub.41}

%% see: notes/note_DG_cherries_41_knot_13_Nov_2024.pdf

The gluing equations matrix of the default \texttt{SnapPy} triangulation of
the $4_1$ knot is
\begin{equation}
\label{41raw}
\begin{pmatrix}
2 & 1 & 0 & 2 & 1 & 0 \\
0 & 1 & 2 & 0 & 1 & 2 \\
1 & 0 & 0 & 0 & 0 & -1 \\
1 & 1 & 1 & 1 & -1 & -3 
\end{pmatrix}
\end{equation}
hence the three gluing equation matrices are
\begin{equation}
\label{41G}
\mb{G} =
\begin{pmatrix}
  2 & 2 \\
  0 & 0 \\
  1 & 0 \\
  1 & 1 
\end{pmatrix}, \qquad
\mb{G}' =
\begin{pmatrix}
  1 & 1 \\
  1 & 1 \\
  0 & 0 \\
  1 & -1
\end{pmatrix}, \qquad
\mb{G}'' =
\begin{pmatrix}
  0 & 0 \\
  2 & 2 \\
  0 & -1 \\
  1 & -3 
\end{pmatrix}
\qquad
\boldsymbol{\eta}=
\begin{pmatrix}
2 \\ 2 \\ 0 \\ 0
\end{pmatrix}
\,.
\end{equation}
Eliminating the shapes $z'_j$ (instead of $z''_j$ as before), removing the second
edge equation and the longitude equation gives the matrices
\begin{equation}
\label{41ABmu}
A_\mu =
\begin{pmatrix}
  1 & 1 \\
  1 & 0 
\end{pmatrix}, \qquad
B_\mu =
\begin{pmatrix}
  -1 & -1 \\
  0 & -1 \\
\end{pmatrix}, \qquad
\nu_\mu =
\begin{pmatrix}
  0 \\
  0 
\end{pmatrix} 
\end{equation}
with $B_\mu$ unimodular and
$B_\mu^{-1} A_\mu = \begin{pmatrix} 0 & -1 \\ -1 & 0 \end{pmatrix}$.
The flattenings are given by
\begin{equation}
\label{41flat}
f' = (f_1,f_2)^t, \qquad f = (f_2,f_1)^t
\end{equation}
for arbitrary integers $f_1,f_2$.

The geometric solution of the gluing equations is $(z_1,z_2)=(\z_6,\z_6)$
%Let $\a \in \BQ$ denote a rational number with denominator $n$, $\z=\e(\a)$ a primitive $n$th root of unity,
where $\z_6=\e(1/6)$. Then $\th=\z_6^{1/n}=\e(1/(6n))$.
Since $B_\mu$ is invertible over $\BZ$, using Equation~\eqref{amth} with $d=1$,
we obtain that the 1-loop invariant of the $4_1$ at roots of unity with respect
to the meridian $\mu$ is given by
\begin{equation}
\label{tau41M}
\tau_{4_1,\mu}(\z) = \frac{1}{n\sqrt[4]{3}} \calD_{\z^{-1}} (\th^{-1})^2
\sum_{k, \ell \bmod n} \frac{\z^{-k \ell} \th^{k+\ell}}{
  (\z\th^{-1};\z)_k(\z\th^{-1};\z)_\ell}
\end{equation}
where a (fixed) $8$-th root of unity is removed for clarity. This agrees with the
following function of~\cite[Eqn.(95)]{GZ:kashaev} up to a $12n$-th root of unity.
% I cannot find a simple formula for the root of unity.
\begin{equation}
\label{J141}
J^{(\s_1)}(\z) = \frac{1}{\sqrt[4]{3}} \frac{1}{\sqrt{n}}
\calD_{\z}(\z \th) \calD_{\z^{-1}}(\z^{-1} \th^{-1}) \sum_{k \bmod n}
(\z \th;\z)_k (\z^{-1} \th^{-1};\z^{-1})_k \,.
\end{equation}

%% see pari files: BWT-quantum-invariants/pari
%% 41Wmatrix.0.stavros, 41Complete.7.withgraphs.stavros, BWY.3.stavros.gp

%% For a 2-dim sum of the cherries, see pari file:
%% BWT-quantum-invariants/pari/2d.sum.4_1.cherrie.for.stavros.pari

The sum above is motivated by Kashaev's formula for his namesake invariant
of the $4_1$ knot; see~\cite[Eqn.(7.4)]{GZ:kashaev}.

On the other hand, if we remove the second edge equation and the meridian
equation and divide the longitude equation by $2$, we obtain the matrices
\begin{equation}
\label{41ABlambda}
A_\lambda =
\begin{pmatrix}
  1 & 1 \\
  0 & 1 
\end{pmatrix}, \qquad
B_\lambda =
\begin{pmatrix}
  2 & 2 \\
  0 & 2 
\end{pmatrix}, \qquad
\nu_\lambda =
\begin{pmatrix}
  2 \\ 1
\end{pmatrix} 
\end{equation}
with $\tfrac{1}{2} B$ unimodular and
$2B^{-1}_\lambda A_\lambda = \begin{pmatrix} 1 & 0 \\ 0 & 1 \end{pmatrix}$
and $2B^{-1}_\lambda \nu_\lambda = \begin{pmatrix} 1 \\ 1 \end{pmatrix}$.
Equation~\eqref{amth} gives the 1-loop invariant for odd $n$ using the flattening $f'=(-1,1)^t$, $f=(1,0)^t$.
\begin{equation}
\label{tau41L}
  \tau_{4_1,\lambda}(\z) =
  \frac{\calD_{\z^{-1}} (\th^{-1})^2}{n\sqrt{3}\zeta_6^{\frac{1-n}{2n}}} 
\Big( \sum_{k \bmod n} (-1)^k  \frac{\z^{k^2+k/2}
  \th^{-k}}{(\z\th^{-1};\z)_{2k}} \Big)^2 \,.
\end{equation}

%% see Mathematica file: ABMatrices.41.Longitude.nb

\begin{comment}
A direct calculation shows that
\begin{equation}
\label{tau41phase}
\tau_{4_1,\lambda}(\z)/\tau_{4_1,\lambda}(1) = 
%e^{-2\pi \ii(n-1/n)/24}
%\e(-\tfrac{1}{24}(n-\frac{1}{n}))
e^{\pi\ii\frac{1-n}{2n}}
\abs{\tau_{4_1,\lambda}(\z)/\tau_{4_1,\lambda}(1)} \,.
\end{equation}
\end{comment}

\subsection{The BWY invariant for $LR$}
\label{sub.BWY41}

For the definition of the BWY invariant of a pA homeomorphism $\vphi$ of a punctured
surface at roots of unity, we refer the reader to~\cite{BWY:I,BWY:II}. The invariant
was explicitly defined for $q=\e(1/n)$ for an odd positive integer $n$, but it can be
extended to the case of arbitrary roots of unity $q$, discussed in detail in
Sections~\ref{sub.BWYdef} and \ref{sec.even}. We denote the corresponding invariant
by $T_\vphi$ as in Equation~\eqref{TBWY}.

%% subtlety on showing that the pA map is a smooth point of a slice
%% of the character variety, as was pointed out by Thang L\^e. 

For the case of a once-punctured torus there are two distinguished elements $L$ and $R$
of its mapping class group and every element of its mapping class group is
conjugate to a product of a word of $L/R$. 

As an example, the $4_1$ complement is the mapping torus of $LR$. Using
Definition~\ref{def.BWY}, we have
% QQbar.zeta(6)^(-(n-1)/(2*n)) * cyclicD(1/QQbar.zeta(6*n),
% -2*Integer(i)/n, n, prec=100)^-2
\begin{equation}
\label{BWYLR}
T_{LR}(q) = \frac{1}{n} \zeta_6^{\frac{n-1}{2n}} \calD_{q^{-2}}(\th^{-1})^2
\Big( \sum_{k \bmod n} q^{\frac{1}{2}(k^2-k)} (-\th)^{k/2} (\th^{-1};q^{-2})_k \Big)^2
\end{equation}
where $\sqrt{-\theta}$ is chosen so that $(-\theta)^{n/2}=\zeta_6$.

The two formulas~\eqref{tau41L} and~\eqref{BWYLR}, after setting $\z=q^2$,
syntactically agree! Indeed, replace $k$ by $-2k$ in the summand of~\eqref{BWYLR},
and use Equation~\eqref{i1} to move the $q$-Pochhammers from the numerator to the
denominator,
\begin{equation}
q^{\frac{1}{2}(k^2-k)} (-\th)^{k/2} (\th^{-1};q^{-2})_k
\mapsto q^{2k^2+k} (-\th)^{-k} (\th^{-1};q^{-2})_{-2k}
= (-1)^k \frac{q^{2k^2+k} \th^{-k}}{(q^2\th^{-1};q^2)_{2k}} \,.
\end{equation}
Doing so, we obtain the summand of ~\eqref{tau41L} with $\z$ replaced by $q^2$.
In the next section we will see that this is not an accident, in fact it persists
for all pA maps of a once-punctured torus.

\begin{comment}
(-1)^k \z^{2k^2-k} \th^{k} (\th^{-1};\z^{-2})_{2k} =
(-1)^k \frac{\z^{2k^2-k} \th^{k}}{(\z^{-2}\th^{-1};\z^{-2})_{-2k}} \\
\mapsto

The formula given in \cite{BWY:I} involves a lot of choices. One possibility is
\begin{equation}
\label{TLR}
\begin{aligned}
\abs{T_{LR}(\z)} & = \frac{1}{n \abs{\calD_{\z^2}(\th^{-1})}^2}
\Big| \sum_{k \bmod n} (-1)^{k} \th^{k}
\z^{2 k^2 + k}  (\z^{-2}\th^{-1};\z^{-2})_{2k} \Big|^2 \\
& = \frac{1}{n \calD_{\z^2}(\th^{-1}) \calD_{\z^{-2}}(\th)}
\Big( \sum_{k \bmod n} (-1)^{k} \th^{k}
\z^{2 k^2 + k}  (\z^{-2}\th^{-1};\z^{-2})_{2k} \Big)^2 \,.
\end{aligned}
\end{equation}
\end{comment}

%%%%%%%%%%%%%%%%%%%%%%%%%%%%%%%%%%%%%%%%%%%%%%%%%%%%%%%%%%%%%%%%%%%%%%%%%%%%
%%%%%%%%%%%%%%%%%%%%%%%%%%%%%%%%%%%%%%%%%%%%%%%%%%%%%%%%%%%%%%%%%%%%%%%%%%%%

\section{1-loop equals BWY for once-punctured torus bundles}
\label{sub.equal}

In this section we prove Conjecture~\ref{conj.1} for pA homeomorphisms
of once-punctured torus bundles. Some, but not all, of our arguments can be
adapted to the case of punctured surface of negative Euler characteristic, but
for concreteness, we focus on once-punctured surfaces. 

\subsection{Layered triangulations of once-punctured torus bundles}
\label{sub.layered}

Let $\varphi$ be an orientation-preserving pseudo-Anosov homeomorphism of the
once-punctured torus $\ptorus$. It is well known that up to conjugation,
\begin{equation}
\label{vphi}
\vphi = \pm\varphi_1\dotsm\varphi_{N} \,,
\end{equation}
where each $\varphi_i$ is one of two elements $L$ and $R$ which lift to linear
actions of $\begin{pmatrix}1&0\\1&1\end{pmatrix}$ and
$\begin{pmatrix}1&1\\0&1\end{pmatrix}$, respectively, of the $\BZ^2$-covering space
$\BR^2\setminus\BZ^2$ of $\ptorus$. Moreover, both $L$ and $R$ appear in the product.
Note this convention is consistent with \texttt{SnapPy} and~\cite{Gueritaud}, but
opposite of \cite{BWY:I,BWY:II}. The two conventions are related by reversing the
orientation, so the difference is immaterial. The sign in \eqref{vphi} changes the
mapping torus $M_\varphi$, but due to the symmetry of $\ptorus$, the only relevant
difference in this paper is the meridian, which does not appear until the end of the
paper. Thus, we ignore this sign for now. Moreover, we use the convention that the
indices are in $\BZ/N \BZ$.

Given this decomposition of $\varphi$, a layered triangulation with $N$ tetrahedra
$T_1,\dotsc,T_{N}$ can be built for the mapping torus $M_\varphi$. This is discussed
in \cite{Gueritaud}. We use conventions of \texttt{SnapPy}, except the first tetrahedron $T_0$ needs to be relabeled as $T_N$ here.

Each tetrahedron is layered on $\ptorus$ as in Figure~\ref{fig-layered-tetra}, where
opposite sides of the square are identified as usual. Each $\varphi_i$ determines
how the top of $T_{i-1}$ is glued to the bottom of $T_i$. See
Figure~\ref{fig-LR-gluing}.

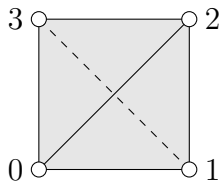
\begin{figure}
\centering
\begin{tikzpicture}
\tikzmath{\r=2;}
\draw[fill=gray!20] (0,0) rectangle (\r,\r);
\draw (0,0) -- (\r,\r);
\draw[dashed] (0,\r) -- (\r,0);
\draw[fill=white,radius=0.1,inner sep=0.2cm] (0,0)circle node[left]{$0$} (\r,0)
circle node[right]{$1$} (\r,\r)circle node[right]{$2$} (0,\r)circle node[left]{$3$};
\end{tikzpicture}
\caption{A tetrahedron layered on the once-punctured torus.}\label{fig-layered-tetra}
\end{figure}

\begin{figure}
\centering
\begin{tikzpicture}[baseline=0cm]
\tikzmath{\r=2;}
\draw[dashed] (0,\r) -- (\r,0);
\begin{scope}[blue]
\draw[thick] (0,\r) -- (0,0) -- (\r,\r) -- (\r,0)
    (0,0) -- (\r/2,\r)  (\r,\r) -- (\r/2,0);
\draw[thick, dashed] (0,0) -- (\r,0)  (0,\r) -- (\r,\r);
%\draw (0,\r) -- (\r,2*\r) -- (\r,\r)  (\r,2*\r) -- (\r/2,\r);
\path[inner sep=0.2cm] (0,0.4)node[left]{$0$} (\r,0.4)node[right]{$1$}
(\r,\r-0.4)node[right]{$2$} (0,\r-0.4)node[left]{$3$};
\end{scope}
\draw[fill=white,radius=0.1,inner sep=0.2cm] (0,0)circle node[left]{$0$}
(\r,0)circle node[right]{$1$} (\r,\r)circle node[right]{$2$} (0,\r)
circle node[left]{$3$};% (\r,2*\r)circle[];
\path (\r/2,-0.8)node{$L$};% (-\r,\r);
\end{tikzpicture}
\qquad
\begin{tikzpicture}[baseline=0cm]
\tikzmath{\r=2;}
\draw[dashed] (0,\r) -- (\r,0);
\begin{scope}[red]
\draw[thick] (0,\r) -- (\r,\r) -- (0,0) -- (\r,0)
    (0,0) -- (\r,\r/2)  (\r,\r) -- (0,\r/2);
\draw[thick, dashed] (0,0) -- (0,\r)  (\r,0) -- (\r,\r);
%\draw (\r,0) -- (2*\r,\r) -- (\r,\r)  (2*\r,\r) -- (\r,\r/2);
\path[inner sep=0.2cm] (0.4,0)node[below]{$0$} (\r-0.4,0)node[below]{$1$}
(\r-0.4,\r)node[above]{$2$} (0.4,\r)node[above]{$3$};
\end{scope}
\draw[fill=white,radius=0.1,inner sep=0.2cm] (0,0)circle node[below]{$0$}
(\r,0)circle node[below]{$1$} (\r,\r)circle node[above]{$2$}
(0,\r)circle node[above]{$3$};% (2*\r,\r)circle[];
\path (\r/2,-0.8)node{$R$};
\end{tikzpicture}
\caption{Layering of $L$ and $R$.}\label{fig-LR-gluing}
\end{figure}
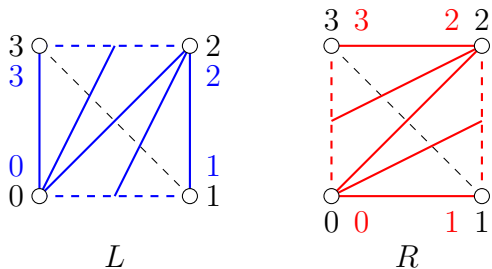

The gluing equations can be obtained by looking at the cusp. For a single
tetrahedron, this looks like Figure~\ref{fig-tetra-cusp} from the outside. When
the next tetrahedron is layered on top, this looks like Figure~\ref{fig-cusp-layer}.

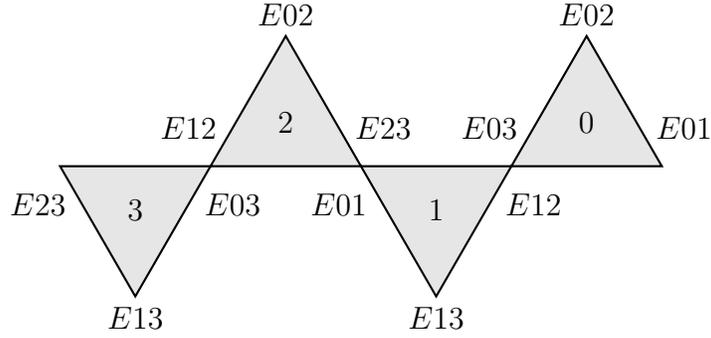
\begin{figure}
\centering
\begin{tikzpicture}
\tikzmath{\r=2; \h=\r/2/sqrt(3);}
\draw[thick,fill=gray!20] (0,0) -- ++(-60:\r) -- ++(60:2*\r) -- ++(-60:2*\r)
-- ++(60:2*\r) -- ++(-60:\r) -- cycle;
\path (0.5*\r,-\h)node{$3$} (1.5*\r,\h)node{$2$} (2.5*\r,-\h)node{$1$}
(3.5*\r,\h)node{$0$};
\path (0.5*\r,-3*\h)node[below]{$E13$} (2.5*\r,-3*\h)node[below]{$E13$};
\path (1.5*\r,3*\h)node[above]{$E02$} (3.5*\r,3*\h)node[above]{$E02$};
\begin{scope}[every node/.style={circle}]
\path (4*\r,0)node[anchor=-120]{$E01$} (0,0)node[anchor=60]{$E23$};
\path (3*\r,0)node[anchor=-60]{$E03$} node[anchor=120]{$E12$};
\path (2*\r,0)node[anchor=60]{$E01$} node[anchor=-120]{$E23$};
\path (\r,0)node[anchor=-60]{$E12$} node[anchor=120]{$E03$};
\end{scope}
\end{tikzpicture}
\caption{Triangles of the same tetrahedron on the cusp.}\label{fig-tetra-cusp}
\end{figure}

\begin{figure}
\centering
\begin{tikzpicture}
\tikzmath{\r=1.5; \h=\r/2/sqrt(3);}
\fill[blue!20] (0,2*\h) -- (\r,0) -- ++(0,4*\h) -- (3*\r,0) -- ++(0,4*\h)
-- (4*\r,2*\h) -- (4*\r, 0) -| cycle;
\draw[thick,fill=gray!20] (0,0) -- ++(-60:\r) -- ++(60:2*\r) -- ++(-60:2*\r)
-- ++(60:2*\r) -- ++(-60:\r) -- cycle;
\path (0.5*\r,-\h)node{$3$} (1.5*\r,\h)node{$2$} (2.5*\r,-\h)node{$1$}
(3.5*\r,\h)node{$0$};
\begin{scope}[blue]
  \draw[thick] (0,2*\h) -- (\r,0) -- ++(0,4*\h) -- (3*\r,0) -- ++(0,4*\h)
  -- (4*\r,2*\h);
  \path (1/3*\r,2/3*\h)node{$3$} (7/6*\r,2.5*\h)node{$2$} (9/4*\r,2/3*\h)node{$1$}
  (19/6*\r,2.5*\h)node{$0$};
\end{scope}
\path (2*\r,-3.5*\h)node{$L$};
\end{tikzpicture}
\qquad
\begin{tikzpicture}
\tikzmath{\r=1.5; \h=\r/2/sqrt(3);}
\fill[red!20] (0,0) -- (2*\r,4*\h) -- (2*\r,0) -- (4*\r,4*\h) |- cycle;
\draw[thick,fill=gray!20] (0,0) -- ++(-60:\r) -- ++(60:2*\r) -- ++(-60:2*\r)
-- ++(60:2*\r) -- ++(-60:\r) -- cycle;
\path (0.5*\r,-\h)node{$3$} (1.5*\r,\h)node{$2$} (2.5*\r,-\h)node{$1$}
(3.5*\r,\h)node{$0$};
\begin{scope}[red]
\draw[thick] (0,0) -- (2*\r,4*\h) -- (2*\r,0) -- (4*\r,4*\h) -- (4*\r,0);
\path (3/4*\r,2/3*\h)node{$3$} (11/6*\r,2.5*\h)node{$2$} (11/4*\r,2/3*\h)node{$1$}
(23/6*\r,2.5*\h)node{$0$};
\end{scope}
\path (2*\r,-3.5*\h)node{$R$};
\end{tikzpicture}
\caption{Layering tetrahedra on the cusp.}\label{fig-cusp-layer}
\end{figure}
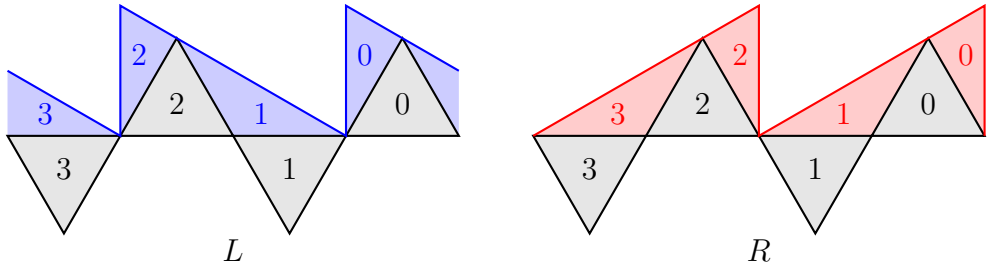

Now let $E_i$ be the $E02$ edge of $T_{i-1}$. Suppose $\varphi_i=L$, and the next
time $L$ appears at $\varphi_{i+k}$. (Recall the indices are cyclic.) Using the
layering rules of the cusp, we see that $E_i$ is identified with $E01$ and $E23$
of $T_i,\dotsc,T_{i+k-1}$ and topped off with $E13$ of $T_{i+k}$. See
Figure~\ref{fig-edge-gluing} for an example where $k=3$. This shows that the
gluing equation at edge $E_i$ is
\begin{equation}
\label{eq-gluingL}
z'_{i-1}z_i^2\dotsm z_{i+k-1}^2z'_{i+k}=e^{2\pi\ii}.
\end{equation}
The case of $\varphi_i=R$ can be obtained similarly, giving the equation
\begin{equation}
\label{eq-gluingR}
z'_{i-1}(z''_i)^2\dotsm (z''_{i+k-1})^2z'_{i+k}=e^{2\pi\ii}.
\end{equation}

\begin{figure}
\centering
\begin{tikzpicture}[thick]
\tikzmath{\r=2; \h=\r/2/sqrt(3);}
% T_{i+k} final L
\fill[blue!20] (2/3*\r, 7*\h) -- (60:\r) -- (2*\r,4*\h);
\draw[blue] (2/3*\r, 7*\h) -- (2*\r,4*\h)  (19/18*\r,14/3*\h)node{$1$};
% T_{i+2} second R
\fill[red!30] (\r/4,6*\h) -- (2/3*\r, 7*\h) -- (60:\r) -- cycle  (60:\r)
-- (2*\r,4*\h) -- (2*\r,2*\h) -- cycle;
\draw[red] (\r/4,6*\h) -- (2/3*\r, 7*\h) -- (60:\r) -- (2*\r,4*\h)
(17/36*\r,16/3*\h)node{$2$} (3/2*\r,3*\h)node{$1$};
% T_{i+1} first R
\fill[red!20] (0,4*\h) -- (\r/4, 6*\h) -- (60:\r) -- cycle  (60:\r) -- (2*\r,2*\h)
-- (2*\r,0);
\draw[red] (0,4*\h) -- (\r/4, 6*\h) -- (60:\r) -- (2*\r,2*\h)
(\r/4,13/3*\h)node{$2$} (3/2*\r,5/3*\h)node{$1$};
% T_i
\draw[blue,fill=blue!20] (0,0) -- (0,4*\h) -- (2*\r,0)  (\r/6,2.5*\h)node{$2$}
(5/4*\r,2/3*\h)node{$1$};
% T_{i-1}
\draw[fill=gray!20] (0,0) -- (\r,0) -- (60:\r) -- cycle;
\draw (\r,0) -- (2*\r,0);
\path (\r/2,\h)node{$2$};
\end{tikzpicture}
\caption{The edge $E_i$ viewed from the cusp for $\varphi_i=L$.}\label{fig-edge-gluing}
\end{figure}
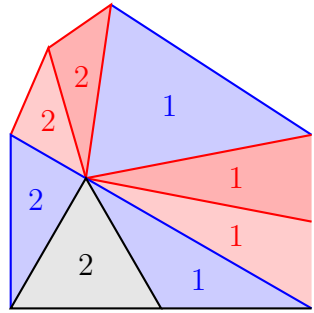

We also need the longitude equation. Note the longitude of the mapping torus is
the peripheral curve of the surface, which appears horizontal in our cusp diagrams.
To obtain the simplest equation possible, we use a cyclic permutation to make
$\varphi_1=L$ and $\varphi_N=R$. Then the region formed by $T_{N-1},T_{N},T_1$
in the cusp contains a longitude. See Figure~\ref{fig-cusp-longi}. The longitude
equation is easily read from the diagram as
\begin{equation}
\label{eq-gluing-longi}
\left(z_{N}(z'_{N-1})^{-1}(z''_{N})^{-1}z'_1\right)^2=e^{0\pi\ii}.
\end{equation}

\begin{figure}
\centering
\begin{tikzpicture}
\tikzmath{\r=1.5; \h=\r/2/sqrt(3);}
% L
\fill[blue!20] (0,0) -- (0,4*\h) -- (3/2*\r,3*\h) -- ++(0,2*\h) -- ++(-30:2*\h)
-- (7/2*\r,3*\h) -- ++(0,2*\h) -- ++(-30:2*\h) |- cycle;
\begin{scope}[blue]
  \draw[blue,thick] (0,4*\h) -- (3/2*\r,3*\h) -- ++(0,2*\h) -- ++(-30:2*\h)
  -- (7/2*\r,3*\h) -- ++(0,2*\h) -- ++(-30:2*\h);
\path (1/2*\r,7/3*\h)node{$3$} (5/2*\r,7/3*\h)node{$1$};
\end{scope}
% R
\fill[red!20] (0,0) -- (2*\r,4*\h) -- (2*\r,0) -- (4*\r,4*\h) |- cycle;
\begin{scope}[red]
\draw[thick] (0,0) -- (2*\r,4*\h) -- (2*\r,0) -- (4*\r,4*\h) -- (4*\r,0);
\draw[dashed] (0,0) -- (0,4*\h);
\path (3/4*\r,2/3*\h)node{$3$} (11/6*\r,2.5*\h)node{$2$} (11/4*\r,2/3*\h)node{$1$}
(23/6*\r,2.5*\h)node{$0$};
\end{scope}
% T_{N-2}
\draw[thick,fill=gray!20] (0,0) -- ++(-60:\r) -- ++(60:2*\r) -- ++(-60:2*\r)
-- ++(60:2*\r) -- ++(-60:\r) -- cycle;
\path (0.5*\r,-\h)node{$3$} (1.5*\r,\h)node{$2$} (2.5*\r,-\h)node{$1$}
(3.5*\r,\h)node{$0$};
\draw[->] (4.5*\r,1.5*\h) -- ++(-5*\r,0) node[above right]{$\lambda$};
\end{tikzpicture}
\caption{A neighborhood of the longitude.}\label{fig-cusp-longi}
\end{figure}
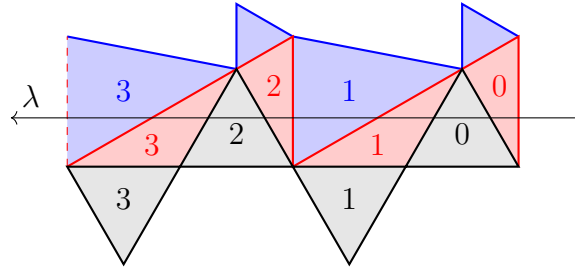

\subsection{Neumann--Zagier data}
\label{sub-layered-NZ}

For layered triangulations of $\ptorus$, the NZ data have very simple forms. Using
Equations \eqref{eq-gluingL}, \eqref{eq-gluingR}, \eqref{eq-gluing-longi}, we have
the following:
\begin{enumerate}
\item
  If $\varphi_i=L$, and the next time $L$ appears at $\varphi_{i+k}$, then
\begin{enumerate}
\item
  $A_{i,i-1}=A_{i,i+k}=1$, and all other entries on row $i$ are $0$.
\item
  $B_{i,i}=B_{i,i+1}=\dotsb=B_{i,i+k-1}=2$, and all other entries on row $i$ are $0$.
\item
  $\nu_i=2$.
\end{enumerate}
\item
  If $\varphi_i=R$, and the next time $R$ appears at $\varphi_{i+k}$, then
\begin{enumerate}
\item
  $A_{i,i-1}=A_{i,i+k}=1$, $A_{i,i}=\dotsb=A_{i+k-1}=-2$, and all other
  entries on row $i$ are $0$.
\item
  $B_{i,i}=B_{i,i+1}=\dotsb=B_{i,i+k-1}=-2$, and all other entries on row $i$ are $0$.
\item
  $\nu_i=2-2k$.
\end{enumerate}
\item
  If $i=N$, the formulas above are replaced with the longitude, which has
  $A_{N,N-1}=-1$, $A_{N,N}=A_{N,1}=1$, $B_{N,N}=2$, and $\nu_{N}=1$.
\item
  In case the indices wrap around and the corresponding entry appears multiple
  times above, then the corresponding formulas add together.
\end{enumerate}

\begin{example}
\label{AB2}
The $(A,B,\nu)$ data of $LR$ and $LLR$ are given by 
\begin{equation}
\label{ABnuLR}
A_{LR} =
\begin{pmatrix}
  1 & 1 \\
  1-1 & 1 
\end{pmatrix}, \qquad
B_{LR} =
\begin{pmatrix}
  2 & 2 \\
  0 & 2 
\end{pmatrix}, \qquad
\nu_{LR} =
\begin{pmatrix}
  2 \\ 1
\end{pmatrix} 
\end{equation}
(which matches with~\eqref{41ABlambda}) and
\begin{equation}
\label{ABnuLLR}
A_{LLR} =
\begin{pmatrix}
  0 & 1 & 1 \\
  1+1 & 0 & 0 \\
  1 & -1 & 1
\end{pmatrix}, \qquad
B_{LLR} =
\begin{pmatrix}
  2 & 0 & 0 \\
  0 & 2 & 2 \\
  0 & 0 & 2
\end{pmatrix}, \qquad
\nu_{LLR} =
\begin{pmatrix}
  2 \\ 2 \\ 1
\end{pmatrix} \,.
\end{equation}
\end{example}

It is easy to see that $\frac{1}{2}B$ is unimodular since it is upper triangular
with $\pm1$'s on the diagonal. We define
\begin{equation}
\label{Peta}
Q:=2B^{-1}A, \qquad \eta:=2B^{-1}\nu \,.
\end{equation}
The Neumann--Zagier equations now read
\begin{equation}\label{eq-NZ-Q}
\Big(\sum_{j=1}^N Q_{ij}\log z'_j\Big)+2\log z_i=\pi\ii\eta_i
\qquad\text{or}\quad
\Big(\prod_{j=1}^N z_j^{\prime Q_{ij}}\Big)z_i^2=(-1)^{\eta_i}.
\end{equation}

\begin{lemma}
\label{lem.Pm}  
$\eta_i$ is the number of $L$'s in $(\varphi_i,\varphi_{i+1})$, and $Q$ is symmetric
with the $i$-th column having zero entries except at $i-1,i,i+1$ given by
\begin{equation}
\begin{split}
Q_{i-1,i}&=Q_{i,i-1}=
\begin{cases}
1,&\varphi_i=L,\\-1,&\varphi_i=R,
\end{cases}\\
Q_{i,i}&=\text{number of $R$'s in }(\varphi_i,\varphi_{i+1}).
\end{split}
\end{equation}
\end{lemma}

\begin{proof}
Direct calculation.
\end{proof}

\begin{corollary}
We have: $Q1=\eta$.
\end{corollary}

\subsection{The Chekhov--Fock algebra}
\label{sub.CF}

For the moment, $q\in\cx^\times$ is any nonzero number. The Chekhov--Fock
algebra~\cite{FC} of the once-punctured torus $\ptorus$ is the quantum torus
\begin{equation}\label{eq-CF-def}
\begin{split}
\qtorus&=\BC\langle X^{\pm1},Y^{\pm1},Z^{\pm1}\rangle / \ideal{XY-q^4YX,YZ-q^4ZY,ZX-q^4XZ} \\
&\cong\cx[P^{\pm1}]\langle X^{\pm1},Y^{\pm1}\rangle / \ideal{XY-q^4YX}\,.
\end{split}
\end{equation}
Here, $P:=[XYZ]=q^{-2}XYZ$ is central, where the bracket denotes Weyl-ordering. The
generators $X,Y,Z$ are associated to the edges of a triangulation of $\ptorus$ in a
way such that $X,Y,Z$ appear counterclockwise around both triangles, and $P$ is
associated to the puncture. (This is opposite of \cite{BWY:I} to account for the
opposite choice of $L,R$.) Note that all triangulations of $\ptorus$ are
combinatorially equivalent, but the Chekhov--Fock algebras are related in a
non-trivial way. Let $\lambda_i$ denote the triangulation of $\ptorus$ made out of
the top faces of $T_i$. See the solid lines of Figure~\ref{fig-layered-tetra}.  The
Chekhov--Fock algebra of $\lambda_i$ is denoted $\qtorus_i$, and generators of
$\qtorus_i$ are denoted with the subscript $i$ as well. We choose $X_i$ to be the
edge $E02$ of $T_i$, which determines $Y_i$ to be edge $E01=E23$ and $Z_i$ to be
$E03=E12$.

There is a family of
isomorphisms $\Phi_{ji}:\qteich_i\to\qteich_j$ connecting the division algebras
(i.e., skew-fields) $\qteich_i$ of the Chekhov--Fock algebras. They satisfy the
cocycle conditions $\Phi_{ii}=\id$ and $\Phi_{kj}\circ\Phi_{ji}=\Phi_{ki}$, so it
suffices to describe $\Phi_{i-1,i}$. The explicit formulas are
\begin{equation}
\label{eq-CF-flip}
\begin{aligned}
\Phi_{i-1,i}(P_i)&=P_{i-1}.\\
\Phi_{i-1,i}(X_i)&=
\begin{cases}
Y_{i-1}^{-1},&\varphi_i=L,\\
Z_{i-1}^{-1},&\varphi_i=R,
\end{cases}\\
\Phi_{i-1,i}(Y_i)&=
\begin{cases}
(1+qY_{i-1})(1+q^3Y_{i-1})X_{i-1},&\varphi_i=L,\\
(1+qZ_{i-1})(1+q^3Z_{i-1})Y_{i-1},&\varphi_i=R\,.
\end{cases}
\end{aligned}
\end{equation}

The discussion above works for all invertible $q$, but now we need to specialize to
roots of unity of odd order $n$. We will keep the notation $q$, since we need to set
$\zeta=q^2$.

Let $\{w_k\}_{k\in\ints/n\ints}$ be some fixed basis of $\cx^n$. Define two linear
operators $S,T\in\End(\cx^n)$ by
\begin{equation}\label{eq-ST-def}
Sw_k=q^kw_k,\qquad Tw_k=w_{k+1}.
\end{equation}
It is easy to check that $S^n=T^n=\id$ and $ST=qTS$.

The center of $\qtorus$ is generated by $X^n,Y^n$, and $P$. Every finite dimensional
irreducible representation of $\qtorus$ has dimension $n$ and is uniquely determined
by the central elements up to isomorphism, which has the form
$\rho_i:\qtorus_i\to\End(\BC^n)$ with
\begin{equation}\label{rhodef}
\rho_i(P_i)=p_i\id,\quad
\rho_i(X_i)=a_iS^2,\quad
\rho_i(Y_i)=b_iT^2,\quad
\rho_i(Z_i)=c_iq^2(S^2T^2)^{-1}.
\end{equation}
Here, $p_i,a_i,b_i,c_i\in\BC^\times$ are constants satisfying $a_ib_ic_i=p_i$. When
we match this with the layered triangulation of the mapping torus, $-a_i^n$ is
identified with $z'_i$ due to the cross-ratio interpretation on both sides, and
$p_i^n$ is the eigenvalue squared of the longitude, which is $1$ for the complete
hyperbolic structure.

\subsection{Definition of the BWY invariant}
\label{sub.BWYdef}

The compatibility conditions between $\rho_i$ and $\Phi_{ji}$ are given
in \cite[Prop.~23]{BWY:I}. Although they gave a method of choosing compatible
constants, it does not match well with the 3-dimensional picture, so we give an
alternative definition.

Recall the discrete Fourier transform whose kernel is given by the matrix
\begin{equation}\label{eq-fouL}
\fourier_L = \frac{1}{\sqrt{n}}(q^{ij})_{i,j\in\ints/n\ints}
\end{equation}
where $q$ is a root of unity of odd order $n$. 
It is well known that $\fourier$ is unitary and $\fourier^4 = 1$. Define a related matrix
\begin{equation}\label{eq-fouR}
\fourier_R=\frac{1}{\sqrt{n}}(q^{\frac{1}{2}(i-j)^2})_{i,j\in\ints/n\ints},
\end{equation}
where $\frac{1}{2}$ is interpreted as $2^{-1}\bmod{n}$ as before.

\begin{comment}
\begin{split}
\fourier_L&=\fourier,\qquad \fourier_R=S_R\fourier^{-1}S_R,\quad
\text{where }S_R=\diag(q^{2k^2})_{k=0}^{n-1},\\
\end{comment}

Now define the following matrices
\begin{equation}
\label{eq-BWY-mat}
D_i=\diag(d_i^k(-q^{-1}a_i^{-1};q^{-2})_k)_{k\in\ints/n\ints},\qquad
d_i^2=
\begin{cases}
a_{i-1}b_i^{-1},&\varphi_i=L,\\
%a_{i-1}^{-1}a_i^2c_i
b_{i-1}b_i^{-1},&\varphi_i=R.
\end{cases}
\end{equation}
The choice of the square root $d_i$ is discussed later. Then we define $H_i=\fourier_{\varphi_i}D_i$.

\begin{lemma}\label{lem.H}
Assume that $p_i=p$ is independent of $i$, $D_i$ is well-defined, i.e.,
\begin{equation}
\label{eq-Dperiodic}
d_i^n(-q^{-1}a_i^{-1};q^{-2})_n=1,
\end{equation}
and
\begin{equation}
\label{abc2}
a_i = \begin{cases}
  b_{i-1}^{-1}, & \text{if } \varphi_i=L \,, \\
  c_{i-1}^{-1}, & \text{if } \varphi_i=R \,.
\end{cases}
\end{equation}
%$a_i=b_{i-1}^{-1}$ if $\varphi_i=L$, $a_i=c_{i-1}^{-1}$ if $\varphi_i=R$.
Then
\begin{equation}\label{eq-compat-rep}
\rho_i(r)=H_i^{-1}\cdot\left(\rho_{i-1}\circ\Phi_{i-1,i}(r)\right)\cdot H_i
\end{equation}
for all $r\in\qtorus_i$, and with $H=H_1H_1\dotsm H_{N}$,
\begin{equation}
\rho_N(r)=H^{-1}\cdot(\rho_0\circ\Phi_{0,N}(r))\cdot H.
\end{equation}
\end{lemma}

A technicality here is that $\Phi_{i-1,i}(r)$ is not in $\qtorus_i$ but in a
localization. The set of denominators can be deduced from \eqref{eq-CF-flip}. The
lemma implicitly claims that $\rho_{i-1}$ can be (uniquely) extended to this
localization, which follows easily from the calculations in the proof.

\begin{proof}
The equality is trivial for $P_i$ which maps to $p_i\id$. For $r=X_i$, we use
\eqref{abc2} and the following identities that can be verified directly
\begin{equation}
\fourier_L^{-1}T^{-1}\fourier_L=S,\qquad
\fourier_R^{-1}(q^{-1/2}ST)\fourier_R=S.
\end{equation}
For $r=Y_i$, using two additional identities
\begin{equation}
\fourier_L^{-1}S\fourier_L=T,\qquad
\fourier_R^{-1}T\fourier_R=T,
\end{equation}
we get
\begin{equation}
\fourier_{\varphi_i}^{-1}\cdot\left(\rho_{i-1}\circ\Phi_{i-1,i}(Y_i)\right)\cdot \fourier_{\varphi_i}=(1+qa_i^{-1}S^{-2})(1+q^3a_i^{-1}S^{-2})d_i^2b_iT^2.
\end{equation}
Then it is simple to check that \eqref{eq-compat-rep} holds for $r=Y_i$ using the
definition of $D_i$.
\end{proof}

\begin{definition}
\label{def.BWY}
The BWY invariant at the complete hyperbolic structure is given by
\begin{equation}
\label{Tphidef}
T_{\varphi,m}: \mu'_\BC \to \overline{\BQ}/\mu'_\BC, \qquad 
T_{\varphi,m}(q)=\tr(H)/\det(H)^{1/n} 
\end{equation}
where the constants used in the definition of $H$ above are given by
\begin{align}
p_i=q^{2m},\qquad
a_i=-q^{-1}\theta_i\quad\text{where }\theta_i&=\exp(\frac{1}{n}\log z'_i),\\
d_i=q^{m\beta_i-\frac{1}{2}\eta_i}\exp\left(\frac{1-n^2}{2n}\pi\ii\eta_i-\frac{1}{n}\log z_i\right)
\quad\text{where }
\beta_i&=\begin{cases}
-1,&\varphi_i\varphi_{i+1}=LR,\\
1,&\varphi_i\varphi_{i+1}=RL,\\
0,&\text{otherwise}
\end{cases}
\end{align}
for all $i=1,\dots,N$. Here, $\frac{1}{2}$ in the exponent of $q$ means $2^{-1}\bmod{n}$ as before.

$m\in\ints/n\ints$ is the descendant index. As with the 1-loop invariants, we mainly
consider the case $m=0$, in which case we omit it from the notation.
\end{definition}

The periodicity \eqref{eq-Dperiodic} is easily checked using
\begin{equation}
d_i^n=z_i^{-1},\qquad
(-q^{-1}a_i^{-1};q^{-2})_n=(\theta_i^{-1};q^{-2})_n=1-z^{\prime-1}_i.
\end{equation}
To satisfy the rest of Lemma~\ref{lem.H}, we use \eqref{abc2} and the conservation
condition $p_i=a_i b_i c_i=q^{2m}$ to recover
\begin{equation}\label{abc}
b_i=\begin{cases}
a_{i+1}^{-1},&\text{if }\varphi_{i+1}=L,\\
q^{2m}a_i^{-1}a_{i+1},&\text{if }\varphi_{i+1}=R,
\end{cases} \qquad
c_i=\begin{cases}
q^{2m}a_i^{-1}a_{i+1},&\text{if }\varphi_{i+1}=L,\\
a_{i+1}^{-1},&\text{if }\varphi_{i+1}=R,
\end{cases}
\end{equation}
Using the Neumann--Zagier equations \eqref{eq-NZ-Q}, we get
\begin{equation}\label{di}
d_i^2=q^{2m\beta_i}(-q)^{-\eta_i}\prod_{j=1}^{N}\theta_j^{Q_{ij}},
\end{equation}
which is consistent with the previous definition \eqref{eq-BWY-mat}.
%, and the sign is chosen so that $\frac{1\pm n}{2}$ is even.

\begin{remark}
\label{rem.BWY}
We complement the above definition with some remarks. 

\begin{enumerate}[1., leftmargin=0pt, itemindent=*]
\item
The invariant has a symmetry $m\leftrightarrow-m$. This is not obvious from the
definition here, but it can be explained by an equivalent definition using the skein
algebra.
\item
BWY only consider the absolute value of $T_\vphi$, not $T_\vphi$ itself, due to the
ambiguity of the $n$-th root. From the point of view of asymptotic expansions and
the arithmetic nature of their coefficients, it is unnatural to use the absolute
value. We expect that there is a way to choose a canonical root.
\item
The BWY construction does not reflect the symmetry between $L$ and $R$;
compare~\cite[Equations~(3--4)]{BWY:I} with~\eqref{abc}, keeping in mind that our
$(a_i,b_i,c_i)$ are BWY's $(x_i,y_i,z_i)$.
\item
The definition above manifestly works for all complex roots
of unity with odd denominator, as opposed to only $e^{2 \pi \ii/n}$ for odd $n$ in
certain formulas of BWY. This is a crucial aspect of the Quantum Modularity
Conjecture.
% \item
% A closer look at the consistency shows that we only need $p_i^n=1$. Different choices
% of $p_i$ are called descendants and discussed in Section~\ref{sub.desc2} below.
\end{enumerate}
\end{remark}

\subsection{Proof of Theorem~\ref{thm.1}}
\label{sub.thm1}

In this section we prove Theorem~\ref{thm.1}. The comparison between 1-loop and BWY invariants is stronger than the statement there.

\begin{proposition}\label{prop-BWY-trdet}
With the notations of Definition~\ref{def.1loop} and Section~\ref{sub.BWYdef},
\begin{equation}
\tr H=\sum_{k\in(\ints/n\ints)^N}\frac{a_{k,m}(\theta)}{n^{N/2}},\qquad
\det H=\omega\prod_{i=1}^{N}z_i^{\frac{n-1}{2}}\calD^{-n}_{\zeta^{-1}}(\theta_i^{-1})
\end{equation}
where $q=\e(a/n)$, $\zeta=q^2$, and $\omega$ is a root of unity given by
\begin{equation}
\omega=\left(\frac{-2}{n}\right)^Ne^{-2\pi\ii(2\#L+\#R)ns(-2a,n)},
\end{equation}
where $\left(\frac{c}{d}\right)$ is the Jacobi symbol.
\end{proposition}

The denominator of $s(a,n)$ is at most $2n(3,n)$ (see
e.g.,~\cite[72.Lem.A]{Rademacher}), so $\omega$ is at most a 6th root of unity. Then
Theorem~\ref{thm.1} follows from this using the flattening $f'=1,f=0$.

\begin{proof}
To prove the trace part, we write out the product definition $H=H_1\dotsm
H_N$
\begin{equation}
\tr H
=\sum_{k\in(\ints/n\ints)^N}\prod_{i=1}^N(H_i)_{k_{i-1}k_i}
=\sum_{k\in(\ints/n\ints)^N}\prod_{i=1}^N(\fourier_{\varphi_i})_{k_{i-1}k_i}d_i^{k_i}(\theta_i^{-1};q^{-2})_{k_i}.
\end{equation}
Here, we let $k_0=k_N$ for convenience. By definition,
$(\fourier_{\varphi_i})_{k_{i-1}k_i}=\frac{1}{\sqrt{n}}q^{Q_{\varphi_i}(k)}$ for some
quadratic forms $Q_L,Q_R$. A simple term-by-term calculation shows that
$Q=\sum_{i=1}^N Q_{\varphi_i}$, so the product of $\fourier_{\varphi_i}$ matrix
elements simplifies to $\frac{1}{n^{N/2}}q^{\frac{1}{2}k^tQk}$. Just like
Subsection~\ref{sub.BWY41}, let $k\to-2k$, and use \eqref{i1} and \eqref{di} to get
\begin{equation}
\begin{split}
\tr H
&=\frac{1}{n^{N/2}}\sum_{k\in(\ints/n\ints)^N}\zeta^{k^tQk}\prod_{i=1}^N\frac{q^{-2mk_i\beta_i}(-q)^{k_i\eta_i}\prod_{j=1}^{N}\theta_j^{-k_iQ_{ij}}}{(\zeta\theta_i^{-1};\zeta)_{2k_i}}\\
&=\frac{1}{n^{N/2}}\sum_{k\in(\ints/n\ints)^N}(-1)^{k^t\eta}\zeta^{-mk^t\beta}\zeta^{k^tQk+\frac{1}{2}k^t\eta}\prod_{i=1}^N\frac{\theta_i^{-(Qk)_i}}{(\zeta\theta_i^{-1};\zeta)_{2k_i}}.
\end{split}
\end{equation}
A simple calculation also shows that $(\frac{1}{2}B)\beta=e_N$. Then together with
the definitions of $Q$ and $\eta$, we have an exact match with the sum of
$a_{k,m}(\theta)$.

Now we evaluate the determinant. First, we look at the Fourier matrices
$\fourier_L,\fourier_R$. Note we already know that $\det\fourier$ is a 4-th root of
unity since $\fourier^4=1$, and thus $\det\fourier_R$ is at worst a 12-th root of
unity. We can get explicit formulas in terms of Dedekind sums.

\begin{lemma}
For $n$ odd and $(a,n)=1$,
\begin{equation}
6n \,s(a,n) = \begin{cases} 0 \bmod 3 & \text{if} \,\, (n,3)=1 \\
  a \bmod 3 & \text{otherwise.}
  \end{cases}
\end{equation}
%$6n \,s(a,n)$ is $0\bmod{3}$ if $(n,3)=1$ and $a\bmod{3}$ otherwise.
\end{lemma}

\begin{proof}
If $(n,3)=1$, then the denominator of $s(a,n)$ is $2n$ at worst, so $6n\, s(a,n)$ is
$0 \bmod{3}$.

On the other hand, if $n$ is divisible by $3$, then we have~\cite[72.Lem.B]{Rademacher}
\begin{equation}
12an\, s(a,n)\equiv a^2+1\bmod{3n}.
\end{equation}
We drop the $n$ from the modulus. Then
$a^2\equiv1\bmod{3}$ since $(a,3)=1$. Thus, $12an \,s(a,n)\equiv2a^2\pmod{3}$, which
implies our lemma.
\end{proof}

As a simple corollary, for $n$ odd and $(a,n)=1$, $\frac{a}{n}\sum_{i=1}^{n-1}i^2+2ns(-2a,n)$ is an integer.

\begin{lemma}
For $q=\e(a/n)$ where $n$ is odd and $(a,n)=1$,
\begin{align}
\label{detFLR}
\det\fourier_L&=\left(\frac{-2}{n}\right)e^{-3\pi\ii ns(-2a,n)} \,, 
&
\det\fourier_R&=\left(\frac{-2}{n}\right)e^{-\pi\ii ns(-2a,n)} \,.
\end{align}
\end{lemma}

\begin{proof}
We observe that our $\fourier_L$ can be obtained from the standard Fourier matrix
$\frac{1}{\sqrt{n}}(\zeta^{-ij})$ by a row permutation $i\mapsto-2i$. An extension of
Zolotarev's result (which was originally stated for $n$ prime) shows that the sign of
the permutation is the Jacobi symbol.
% see e.g. arXiv:math/0601026
Thus, we can work with the new matrix instead.

Since the Fourier matrix is a Vandermonde matrix, the determinant is given by the
classical formula
\begin{equation}
\left(\frac{-2}{n}\right)\det\fourier_L=\frac{1}{n^{n/2}}
\prod_{i=1}^{n-1}\prod_{j=0}^{i-1}(\zeta^{-i}-\zeta^{-j}).
\end{equation}
We can pull out factors of $\zeta^{-i}$ and rearrange the product to get
\begin{align}
\left(\frac{-2}{n}\right)\det\fourier_L
&=\frac{1}{n^{n/2}}\zeta^{-\sum_{i=1}^{n-1}i^2}\prod_{k=1}^{n-1}(1-\zeta^{-k})^k
=e^{8\pi\ii ns(-2a,n)}\left(e^{\pi\ii s(-2a,n)}
  \frac{\calD_{\zeta^{-1}}(1)}{\sqrt{n}}\right)^n \,.
\end{align}
Recall $\calD_{\zeta^{-1}}(1)$ is normalized to be $\sqrt{n}$. Then this simplifies
to $e^{-3\pi\ii ns(-2a,n)}$. The second part is similar.
\end{proof}

Next, we calculate $\det D_i$, which is given by
\begin{equation}
\label{eq-detD}
\det D_i=d_i^{n(n-1)/2}\prod_{k=0}^{n-1}(\theta_i^{-1};q^{-2})_k.
\end{equation}
Recall $d_i^n=z_i^{-1}$. A simple reordering of the factors shows that the product of
$q$-Pochhammers in \eqref{eq-detD} is $z_i^{n-1}D_{q^{-2}}^{-1}(\theta_i^{-1})$.
Thus,
\begin{equation}
\det D_i=z_i^{-\frac{n-1}{2}}z_i^{n-1}D_{q^{-2}}^{-1}(\theta_i^{-1})
=z_i^{\frac{n-1}{2}}\left(e^{\pi\ii s(-2a,n)}\calD_{\zeta^{-1}}(\theta_i^{-1})\right)^{-n}.
\end{equation}
Combined with the Fourier matrices above, we obtain the determinant part of
Proposition~\ref{prop-BWY-trdet}.
\end{proof}

% the Dedekind sum $s(a,n)$ for $\z=\e(a/n)$ which is related to the
%  Legendre symbol by
%  \begin{equation}
%  12 n \,s(a,n) = n +1 -2 \Big(\frac{a}{n}\Big) \bmod 8
%  \end{equation}
%
%% see: Rademacher, topics_in_analytic_number_theory_book.pdf p.160.

%%%%%%%%%%%%%%%%%%%%%%%%%%%%%%%%%%%%%%%%%%%%%%%%%%%%%%%%%%%%%%%%%%%%%%%%%%%%
%%%%%%%%%%%%%%%%%%%%%%%%%%%%%%%%%%%%%%%%%%%%%%%%%%%%%%%%%%%%%%%%%%%%%%%%%%%%

\section{Even roots of unity}
\label{sec.even}

As we will see in the next section, even if we only care about the asymptotics of
invariants at odd roots of unity, the Quantum Modularity Conjecture predicts the
appearance of even roots nonetheless. Therefore, we take some time to explicitly
define the BWY invariants at even roots of unity.

The Chekhov--Fock algebra is a quantization of the $\PSL_2(\cx)$-character variety.
There is a related construction that quantizes the $\SL_2(\cx)$-character variety,
which we call the balanced Chekhov--Fock square root algebra or the balanced algebra
for short, and it contains the original Chekhov--Fock algebra.

When $q$ is a root of unity of odd order, the representations of the Chekhov--Fock
algebra and the balanced algebra are essentially the same. This follow from the fact
that irreducible representation of both algebras have equal dimensions. This is
discussed in \cite[Section~3.5]{BWY:I}. However, when the order is even, the
dimensions start to differ between the two algebras. This means there are two closely
related but distinct generalizations of the invariants. It turns out that quantum
modularity selects the one coming from the balanced algebra when the order is a
multiple of $4$.

% Note that in \cite{BWY:I}, the staring point of the invariant is the skein algebra,
% which naturally corresponds to the balanced algebra under the quantum trace map
% \cite{BW:qtrace}, whereas the Chekhov--Fock algebra corresponds to the even part of
% the skein algebra. 

The full generality of the balanced algebra is very technical. Here, we choose to
present the specialized descriptions for the once-punctured torus. Admittedly, some
results are given without proof. We plan to discuss the full theory in later works.

\subsection{Balanced algebra}

Choose a square root $A=q^{1/2}$, which is used in Weyl-ordering. For the punctured
torus $\ptorus$, the balanced algebra has the presentation
\begin{equation}
\label{eq-btorus-def}
\begin{split}
\btorus& =\cx\langle U^{\pm1},V^{\pm1},W^{\pm1}\rangle / \ideal{UV-qVU,VW-qWV,WU-qUW}
\\
&\cong\cx[P^{\pm1}]\langle U^{\pm1},V^{\pm1}\rangle / \ideal{UV-qVU} \,.
%&\cong\cx\langle U^{\pm1},V^{\pm1},W^{\pm1}\rangle / \ideal{UV-qVU,VW-qWV,WU-qUW}\,.
\end{split}
\end{equation}
Here, $P^{-1}=[UVW]=A^{-1}UVW$ is the central element associated to the puncture as
before. $U,V,W$ are associated to pairs of edges, which can be inferred from the
discussion below. The choice of square root $A$ does not change the invariant in the
end, since the automorphism of $\btorus$ sending $U,V,W$ to their negatives
effectively changes the sign of $A$, and it commutes with the constructions below.
Note the symmetry in the second presentation is special to $\ptorus$. Most surfaces
do not have a presentation of the balanced algebra that reflects the symmetry of the
triangulation.

The balanced algebra contains a canonically embedded copy of the original
Chekhov--Fock algebra. The embedding is
\begin{equation}
\qtorus\hookrightarrow\btorus,\qquad
X\mapsto PU^2,\quad Y\mapsto PV^2,\quad
Z\mapsto PW^2,\quad
P\mapsto P.
\end{equation}
Note, $[YZ]^{-1}\mapsto U^2$, $[ZX]^{-1}\mapsto V^2$, and $[XY]^{-1}\mapsto W^2$,
which explains the ``square root'' in the name. There is a balancing condition (which
we do not explain here) that determines which monomials in the Chekhov--Fock algebra
have square roots in the balanced algebra.

As before, the balanced algebra depends on a triangulation, and there is a family of
isomorphisms connecting the division algebras. In the notations of the last section,
for two adjacent triangulations $\lambda_{i-1},\lambda_i$ that are related by a flip,
the isomorphism $\Phi_{i-1,i}:\bteich_i\to\bteich_{i-1}$ is given by
\begin{equation}\label{eq-bl-flip}
\begin{split}
\Phi_{i-1,i}(P_i)&=P_{i-1},\\
\Phi_{i-1,i}(U_i)&=
\begin{cases}
P_{i-1}^{-1}V_{i-1}^{-1},&\varphi_i=L,\\
[U_{i-1}V_{i-1}],&\varphi_i=R,
\end{cases}\\
\Phi_{i-1,i}(V_i)&=
\begin{cases}
(1+qY_{i-1})U_{i-1},&\varphi_i=L,\\
(1+qZ_{i-1})V_{i-1},&\varphi_i=R.
\end{cases}
\end{split}
\end{equation}
These formulas are extensions of \eqref{eq-CF-flip} on the Chekhov--Fock algebras.

\subsection{BWY invariants at all roots of unity}
\label{sub.BWYeven}

Now we assume $q$ is a root of unity of order $n$ with no restriction on $n$ yet. We
have a single description mostly independent of the parity of $n$. This uniformity is
special to the once-punctured torus. Genus 0 surfaces also has a similar property in
terms of the order of $q^2$ instead of $q$.

The center of the balanced algebra is generated by $U^n,V^n,P$ for any $n$. Recall
the operators $S,T\in\End(\cx^n)$ defined in \eqref{eq-ST-def}. Then up to
isomorphism, representations of the balanced algebra are of the form
$\rho_i:\btorus_i\to\End(\BC^n)$ with
\begin{equation}\label{eq-rho-bl}
\rho_i(P_i)=p\id,\qquad
\rho_i(U_i)=u_iS,\qquad
\rho_i(V_i)=v_iT.
\end{equation}
These extend \eqref{rhodef} for suitable choices of constants
$p,u_i,v_i\in\cx^\times$. We have preemptively dropped the dependence of $p$ on $i$.

Lemma~\ref{lem.H} also has an easy generalization with an almost identical proof.
Recall the Fourier matrices $\fourier_L,\fourier_R$ from
\eqref{eq-fouL}--\eqref{eq-fouR} (with the caveat that $q^{1/2}=A$ is now a choice
instead of being determined by $q$ alone). Write $a_i=pu_i^2$, and let
$H_i=\fourier_{\varphi_i}D_i$ where
\begin{equation}\label{eq-BWY-matbl}
D_i=\diag(d_i^k(-q^{-1}a_i^{-1};q^{-2})_k)_{k\in\ints/n\ints},\qquad
d_i=
\begin{cases}
u_{i-1}v_i^{-1},&\varphi_i=L,\\
v_{i-1}v_i^{-1},&\varphi_i=R.
\end{cases}
\end{equation}

\begin{lemma}
Assume that $D_i$ is well-defined, i.e.,
\begin{equation}
\label{eq-Dperiodic-bl}
d_i^n(-q^{-1}a_i^{-1};q^{-2})_n=1,
\end{equation}
and
\begin{equation}
\label{eq-uv}
u_i = \begin{cases}
  (pv_{i-1})^{-1}, & \text{if } \varphi_i=L \,, \\
  u_{i-1}v_{i-1}, & \text{if } \varphi_i=R \,.
\end{cases}
\end{equation}
%$a_i=b_{i-1}^{-1}$ if $\varphi_i=L$, $a_i=c_{i-1}^{-1}$ if $\varphi_i=R$.
Then
\begin{equation}
\rho_i(r)=H_i^{-1}\cdot\left(\rho_{i-1}\circ\Phi_{i-1,i}(r)\right)\cdot H_i
\end{equation}
for all $r\in\btorus_i$.
\end{lemma}

When it comes to the choices of constants, we only need to specify $p$ and $u_i$,
with $v_i$ and $d_i$ being determined by \eqref{eq-uv} and \eqref{eq-BWY-matbl}. The
complication in the previous case where only $d_i^2$ is determined by
\eqref{eq-BWY-mat} is transferred to the choice of $u_i$ since we still need to solve
the periodicity condition \eqref{eq-Dperiodic-bl}. Due to the presence of $q^2$, the
parity of $n$ plays a role in the translation into Neumann--Zagier equations.

As mentioned earlier, \cite{BWY:I} already discussed the case of odd $n$. The choice
of constants depends on an $\SL_2(\cx)$-lift of the hyperbolic structure of the
mapping torus, but the invariant defined from it is independent of the choice.
Setting $p=1$, \eqref{eq-Dperiodic-bl} becomes a square root version of the
Neumann--Zagier equation \eqref{eq-NZ-Q}, and a consistent choice of square roots
corresponds to an $\SL_2(\cx)$-lift. We will not go into more details.

Now we assume $n$ is even. It turns out that the values of $p$ corresponding to the
complete hyperbolic structure of the mapping torus satisfy $(-p)^{n/2}=1$. We
parametrize the values by
\begin{equation}\label{eq-bl-pdef}
p=(-1)^{n/2}q^{2m},
\end{equation}
where $m\in\ints/\frac{n}{2}\ints$ is the descendant index.

Writing $\theta_i=-qa_i$,
the $q$-Pochhammer in \eqref{eq-Dperiodic-bl} simplifies to
$\big(1-\theta_i^{-n/2}\big)^2$. This suggests the identification
$\theta_i^{n/2}=z'_i$, which implies that
\begin{equation}\label{eq-bl-udef}
u_i^n=\left(-q^{-1}p^{-1}\theta_i\right)^{n/2}=-z'_i.
\end{equation}

Rewriting $d_i$ using \eqref{eq-uv} as
\begin{equation}
d_i=p^{\ell_{i+1}}\prod_{j=1}^N u_j^{Q_{ij}}, \qquad
\ell_i=\begin{cases}1,&\varphi_i=L,\\0,&\varphi_i=R.\end{cases}
\end{equation}
Then we can check that the periodicity condition is satisfied using the
Neumann--Zagier equation \eqref{eq-NZ-Q}.

\begin{definition}
The generalization of Definition~\eqref{def.BWY} to roots of unity $q$ of order $n$
divisible by $4$ is given by the same formula \eqref{Tphidef} with the new matrices
$H_i$ above using constants determined from \eqref{eq-bl-pdef} and
\eqref{eq-bl-udef}.
\end{definition}

The omission of $n\equiv2\pmod{4}$ is explained in the next section.

\begin{example}
For $\varphi=LR$, $z'_i=\zeta_6$. The formulas above give
\begin{equation}
\label{eq-BWYLR-even}
T_{LR}(q)=\frac{1}{n}\zeta_6^{\frac{n-2}{2n}}\calD_{q^{-2}}^2(\theta^{-1})\Big(
\sum_{k=0}^{n-1}q^{(k^2-k)/2}(-\theta)^{k/2}(\theta^{-1};q^{-2})_k
\Big)^2
\end{equation}
As mentioned before, the result does not depend on $A=q^{1/2}$. Note the superficial
similarity with \eqref{BWYLR}, with some subtle differences hidden in the notations.

We can also compare with the 1-loop invariant if we pick different elimination
variables to make $B$ unimodular so that $d=1$ in the notation of
Definition~\ref{def.1loop}. This is possible if the homology $H_1(M_\varphi)$ of the
mapping torus $M_\varphi$ has no 2-torsion. In this case, we find numerical
agreements with the generalization of Conjecture~\ref{conj.1} to all roots of unity.
\end{example}

\subsection{Relation to the Chekhov--Fock algebra}

Next, we discuss what happens to the Chekhov--Fock algebra when $n$ is even. Since
only $q^4$ appears in the presentation \eqref{eq-CF-def}, the theory depends on the
order $n'=n/\gcd(n,4)$ of $q^4$ instead. For example, the center of the Chekhov--Fock
algebra is generated by $X^{n'},Y^{n'}$, and $P$. If we focus on the case when $n$ is
even, then we have two possibilities.

If $n\equiv 2\pmod{4}$, then $n'=n/2$ is odd. In this case, the center of the
Chekhov--Fock algebra is the same as the center of the balanced algebra since
$X^{n'}=P^{n'}U^n$ and similarly $Y^{n'}=P^{n'}V^n$. However, the dimension of an
irreducible representation of the Chekhov--Fock algebra is $n'$, which is half of
that of the balanced algebra. Nevertheless, the irreducible representations are very
similar. In fact, an irreducible representation of the balanced algebra decomposes as
the tensor product of an irreducible representation of the Chekhov--Fock algebra and
the 2-dimensional representation of an auxiliary algebra $\marA$ defined by
\cite{Mar:4th}. The algebra $\marA$ is isomorphic to the $2\times2$ matrix algebra,
but it is more naturally described by the presentation
\begin{equation}
\marA=\cx\langle\alpha,\beta\rangle/\ideal{\alpha\beta=-\beta\alpha,\alpha^2=\beta^2=1},
\end{equation}
where $\alpha,\beta$ are associated to some (co)homology classes on $\ptorus$. There
is an algebra embedding (valid for all $q\in\cx^\times$)
\begin{equation}\label{eq-minusq-iso}
\btorus_q\hookrightarrow\btorus_{-q}\otimes\marA,\qquad
P\mapsto -P\otimes1,\quad
U\mapsto -U\otimes\alpha,\quad
V\mapsto -V\otimes\beta,
\end{equation}
where the subscripts of $\btorus$ indicate the commutation coefficient used in the
definition \eqref{eq-btorus-def}. (This is a simpler version of the map defined in
\cite{FKBL:4th}.) An important observation is that it induces an isomorphism of the
Chekhov--Fock algebras $\qtorus_q\cong\qtorus_{-q}$ where generators $X,Y,Z$ are sent
to their negatives. It is easy to find an $SL_2(\ints)$-action on $\marA$ so that
\eqref{eq-minusq-iso} is compatible with flips \eqref{eq-bl-flip}. Since $-q$ has odd
order $n'$, irreducible representations of $\btorus_{-q}\otimes\marA$ have dimension
$2n'=n$, which is the same as that of $\btorus_q$. This implies that the invariant
from $\btorus_q$ factors into the product of $T_{\varphi}(-q)$ from $\btorus_{-q}$
and the invariant from $\marA$. The latter is independent of the order $n$ or the
$PSL_2(\cx)$-character, and it is easily calculable and has absolute value in
$\{0,1,\sqrt{2},2\}$.
%The invariant from $\marA$ is also easily calculable using explicit matrices in $SL_2(\mathbb{Q}(\zeta_8))$ for $L$ and $R$.
In conclusion, neither the Chekhov--Fock algebra nor the balanced algebra at an
$n$-th root where $n\equiv2\pmod{4}$ provides new invariants compared to odd orders,
and in some cases the invariant from the balanced algebra vanishes for all
$n\equiv2\pmod{4}$ since the invariant from $\marA$ can vanish.

The situation where $n$ is divisible by $4$ is less trivial. In this case, the center
of the Chekhov--Fock algebra is generated by the same elements $X^{n'},Y^{n'},P$ as
the previous case, but now it is bigger than the center of the balanced algebra since
$n'=n/4$ is even smaller. An irreducible representation of the balanced algebra
decomposes as the direct sum of 4 irreducible representations of the Chekhov--Fock
algebra, which corresponds to the 4 $\SL_2(\cx)$ lifts of the $\PSL_2(\cx)$-character
of $\ptorus$. The action of the diffeomorphism $\varphi$ permutes the lifts, so the
invariant from the balanced algebra has contributions from lifts that are fixed by
$\varphi$. The individual contributions are related to the invariants from the
Chekhov--Fock algebras, but the determinant of each block is generally not normalized
as 1, only the product of all 4 blocks is normalized as 1 by construction.

%%%%%%%%%%%%%%%%%%%%%%%%%%%%%%%%%%%%%%%%%%%%%%%%%%%%%%%%%%%%%%%%%%%%%%%%%%%%
%%%%%%%%%%%%%%%%%%%%%%%%%%%%%%%%%%%%%%%%%%%%%%%%%%%%%%%%%%%%%%%%%%%%%%%%%%%%

\section{Asymptotics}
\label{sec.asy}

\subsection{Asymptotics and the Quantum Modularity Conjecture} 
\label{sub.QMC}

The quantum modularity conjecture concerns the asymptotics of a square matrix whose
entries are functions $J^{(\s),m}: \BQ \to \BC$, and whose rows are labeled by the
boundary parabolic $\SL_2(\BC)$-representations $\s$ of the cusped hyperbolic
3-manifold, and columns are labeled by integers $m$ (called descendant variables).

Among the boundary parabolic representations there are some distinguished ones:
$\s=\s_0$, the trivial representation, $\s=\s_1$, the geometric representation, and
$\s=\bar\s_1$, the complex conjugate of $\s_1$. The entry $J^{(\s_0),0}$ is none
other than the Kashaev invariant of the cusped hyperbolic 3-manifold. 

Part of the quantum modularity conjecture concerns the asymptotics of $J^{(\s),m}(\ga
X)$ for $\ga = \sma abcd \in \SL_2(\BZ)$ as $X$ goes to infinity with bounded
denominators. Explicitly, Equation (3.6) of~\cite{GZ:kashaev} for $\s=\bar\s_1$
assert that
\begin{equation}
\label{J2asy}
J^{(\bar\s_1)}(\ga X) \sim J^{(\bar\s_1)}(X) e^{\frac{V_\BC}{2\pi i}\big(X+d/c -
\frac{1}{\den(X)^2(X+d/c)}  \big)} \Phi_{a/c}\big(\frac{2 \pi i }{c(cX+d)}\big)
\end{equation}
to all orders in $1/X$. Here $V_\BC = i \mathrm{Vol} + \mathrm{CS} \in \BC/4 \pi^2
\BZ$ is the complexified volume, and $\Phi_{a/c}(h)$ is a power series with algebraic
coefficients, which lie in the trace-field of the knot adjoined $\e(a/c)$ after
divided by the constant term.

\begin{comment}
for $\s=\s_1$ , the asymptotic formula reads
\begin{equation}
\label{J1asy}
J^{(\s_1)}(\ga X) \sim J^{(\s_1)}(X) e^{\frac{V_\BC}{2\pi i}\big(X+d/c +
\frac{1}{\den(X)^2(X+d/c)}  \big)} \Phi_{a/c}\big(\frac{2 \pi i }{c(cX+d)}\big)
\end{equation}
\end{comment}

The Quantum Modularity Conjecture asserts much more than~\eqref{J2asy},
namely includes exponentially small corrections, which when taken into account,
conjecturally define matrix-valued holomorphic functions in the complex cut-plane.

Choosing $\g = \sma 0{-1}10$ and $X=n/2$, with $n$ odd and denoting
$v=V_\BC/(2\pi i)$, Equation~\eqref{J2asy} gives
\begin{equation}
\label{J2asyb}
\frac{J^{(\bar\s_1)}(-2/n)}{J^{(\bar\s_1)}(n/2)} \sim e^{\frac{v}{2} (n-1/n)}
\Phi_0\big(\frac{4 \pi i }{n}\big) \,.
\end{equation}

The above equation is all that we need from the Quantum Modularity Conjecture, and
exactly matches with the numerical asymptotics of the BWY invariant
$T_\varphi(\e(x))$ if it is identified with $J^{(\bar\sigma_1)}(-2x)$ up to some phase
factor, few terms of which are given in~\eqref{LRasyfew} with more terms given in the
sections below.

\subsection{Computing the 1-loop and the BWY invariants}
\label{sub.compute}

In this section we discuss computational aspects of the 1-loop and the BWY invariants.

From its very definition, the computation of the 1-loop invariant at a root of unity
requires $O(n^N)$ steps where $n$ is the order of the root of unity and $N$ is the
number of tetrahedra. Note the $q$-Pochhammers require $O(n)$ time, so the order of
calculation needs to be considered carefully to avoid repeated evaluations.

On the other hand, the BWY invariant of a pA homeomorphism $\vphi$ of a
once-punctured torus bundle is given by the trace of the product of $N$ matrices of
size $n \times n$, where $n$ is the order of the root of unity and $N$ is the length
of $\vphi$ written as a word in $L$/$R$ (see Definition~\eqref{def.BWY}). It follows
that the naive computation of the BWY invariant has time complexity $O(N n^3)$ and
space complexity $O(n^2)$, but this can be optimized. The space requirement can be
lowered to $O(n)$ by splitting the first matrix into row vectors and use
vector-matrix multiplications instead. The time complexity can also be lowered to
$O(Nn^2\log n)$ by a fast Fourier transform implementation.

Note the working precision also affects the complexity. The time is at least linear
in precision, and the space grows linearly in precision. For reference, if $n=1001$
and the precision is 4000 bits (roughly 1200 decimal digits) for both real and
imaginary parts, then a single matrix takes over 1GB of space.

Finally, we remark that catastrophic cancellation is a concern for the numerical
reliability of the result. Experimentally, we find that the precision loss is small
by comparing with results using higher precision.

\subsection{The case of $LR$}
\label{sub.LR}

Using 200 values of $T_{LR}(\e(1/n))$ for odd $n$ from $n=20001,\dots,20399$ and 5000
digit precision of \texttt{pari} and the extrapolation methods of~\cite{GZ:kashaev},
we were able to compute 50 terms of the asymptotics of $T_{LR}(\e(1/n))$. We give 21
terms here and more are available.

%% see pari/BWY_extended.gp

\begin{equation}
\label{LRasy}
\frac{T_{LR}(\e(\frac{1}{n}))}{T_{LR}(\e(-\frac{n}{4}))}
\sim
%\frac{1}{\sqrt{2}} \big(1 - \frac{(-1)^{(n-1)/2}}{\sqrt{3}} \big)
e^{\frac{v}{2}(n-\frac{1}{n})} \Phi_{LR}  \Big(\frac{4 \pi \ii}{n} \Big),
\quad
\Phi_{LR}(h) = \tau_{LR,\lambda}(1)\sum_{k=0}^\infty \frac{a_k}{D_k} \Big(\frac{h}{3 \sqrt{-3}}\Big)^k,
\end{equation}
%% see Mathematica file: ConstantTerm.BWY.LR.nb
% Note: changed to 4\pi i to match \eqref{J2asyb} and GZ
where $\tau_{LR,\lambda}(1)=1/\sqrt{3}$ is the 1-loop invariant at $\zeta=1$, $D_n$
is the universal denominator of~\cite[Eqn(142)]{GZ:kashaev}
\begin{equation}
\label{Dn}
D_n = 2^{3n + v_2(n!)} \prod_{\substack{\text{$p$  prime} \\ p>2}}
\,p^{\;\sum_{i\ge 0}[n/p^i(p-2)]},
\end{equation}
the first 21 of which are given by
%% see pari/BWY_extended.gp
%% vector(7,n,asymD(n,1))
%% in GZ the universal denominators are:
%% 1, 24, 1152, 414720, 39813120, 6688604160, 4815794995200, 115579079884800,
%% in the pari program pari/BWY_extended.gp they are:
%% vector(7,n,asymD(n,0))
%% 1, 12, 288, 51840, 2488320, 209018880, 75246796800, 902961561600
%% differing by power of 2, because the argument has numerator 2\pi i
%% (Tao)
%% but we changed the numerator to 4\pi i now, so the power of 2 is back to GZ
\begin{tiny}
\be
\label{Dvalues}
\begin{aligned}
  D_0 &= 1
  & D_{7} &= 2^{25} \cdot 3^9 \cdot 5^2 \cdot 7
  & D_{14} &= 2^{53} \cdot  3^{19} \cdot  5^{4} \cdot  7^{2} \cdot  11 \cdot  13
  \\
  D_1 &= 2^3 \cdot 3
  & D_{8} &= 2^{31} \cdot  3^{10} \cdot  5^{2} \cdot  7
  & D_{15} &= 2^{56} \cdot  3^{21} \cdot  5^{6} \cdot  7^{3} \cdot  11 \cdot  13
  \cdot  17
  \\
  D_2 &= 2^7 \cdot 3^2
  & D_{9} &= 2^{34} \cdot  3^{13} \cdot  5^{3} \cdot  7 \cdot  11
  & D_{16} &= 2^{63} \cdot  3^{22} \cdot  5^{6} \cdot  7^{3} \cdot  11 \cdot  13
  \cdot  17
  \\
  D_3 &= 2^{10} \cdot 3^4 \cdot 5
  & D_{10} &= 2^{38} \cdot  3^{14} \cdot  5^{3} \cdot  7^{2} \cdot  11
  & D_{17} &= 2^{66} \cdot  3^{23} \cdot  5^{6} \cdot  7^{3} \cdot  11 \cdot  13
  \cdot  17 \cdot 19
  \\
  D_4 &= 2^{15} \cdot 3^5 \cdot 5
  & D_{11} &= 2^{41} \cdot  3^{15} \cdot  5^{3} \cdot  7^{2} \cdot  11 \cdot  13
  & D_{18} &= 2^{70} \cdot  3^{26} \cdot  5^{7} \cdot  7^{3} \cdot  11^{2}
  \cdot  13 \cdot  17 \cdot  19
  \\
  D_5 &= 2^{18} \cdot 3^6 \cdot 5 \cdot 7
  & D_{12} &= 2^{46} \cdot  3^{17} \cdot  5^{4} \cdot  7^{2} \cdot  11 \cdot  13
  & D_{19} &= 2^{73} \cdot  3^{27} \cdot  5^{7} \cdot  7^{3} \cdot  11^{2}
  \cdot  13 \cdot  17 \cdot  19
  \\
  D_6 &= 2^{22} \cdot 3^8 \cdot 5^2 \cdot 7
  & D_{13} &= 2^{49} \cdot  3^{18} \cdot  5^{4} \cdot  7^{2} \cdot  11 \cdot  13
  & D_{20} &= 2^{78} \cdot  3^{28} \cdot  5^{7} \cdot  7^{4} \cdot  11^{2}
  \cdot  13 \cdot  17 \cdot  19
\end{aligned}  
\ee
\end{tiny}
% 1 [[2, 2; 3, 1],
% 2 [2, 5; 3, 2],
% 3 [2, 7; 3, 4; 5, 1],
% 4 [2, 11; 3, 5; 5, 1],
% 5 [2, 13; 3, 6; 5, 1; 7, 1],
% 6 [2, 16; 3, 8; 5, 2; 7, 1],
% 7 [2, 18; 3, 9; 5, 2; 7, 1],
% 8 [2, 23; 3, 10; 5, 2; 7, 1],
% 9 [2, 25; 3, 13; 5, 3; 7, 1; 11, 1],
%10 [2, 28; 3, 14; 5, 3; 7, 2; 11, 1],
%11 [2, 30; 3, 15; 5, 3; 7, 2; 11, 1; 13, 1],
%12 [2, 34; 3, 17; 5, 4; 7, 2; 11, 1; 13, 1],
%13 [2, 36; 3, 18; 5, 4; 7, 2; 11, 1; 13, 1],
%14 [2, 39; 3, 19; 5, 4; 7, 2; 11, 1; 13, 1],
%15 [2, 41; 3, 21; 5, 6; 7, 3; 11, 1; 13, 1; 17, 1],
%16 [2, 47; 3, 22; 5, 6; 7, 3; 11, 1; 13, 1; 17, 1],
%17 [2, 49; 3, 23; 5, 6; 7, 3; 11, 1; 13, 1; 17, 1; 19, 1],
%18 [2, 52; 3, 26; 5, 7; 7, 3; 11, 2; 13, 1; 17, 1; 19, 1],
%19 [2, 54; 3, 27; 5, 7; 7, 3; 11, 2; 13, 1; 17, 1; 19, 1],
%20 [2, 58; 3, 28; 5, 7; 7, 4; 11, 2; 13, 1; 17, 1; 19, 1]]

and the first 21 coefficients $a_k$ are given by
\begin{tiny}
\begin{equation}
\label{LRvalues}
\begin{aligned}
a_0 &= 1
\\
a_1 &= 17
\\
a_2 &= 2305
\\
a_3 &= 4494181
\\
a_4 &= 3330710213
\\
a_5 &= 5712350244311
\\
a_6 &= 52439486675194979
\\
a_7 &= 19266759263233318405
\\
a_8 &= 66121441024491501701765
\\
a_9 &= 16057617271207914483637539331
\\
a_{10} &= 124141789617951906037615282061569
\\
a_{11} &= 990570538120722127305829578974187175
\\
a_{12} &= 40138653318545997972857202310993641324451
\\
a_{13} &= 29576935097999521111492046073898594892534975
\\
a_{14} &= 47226781739778967005629953528286582410693258585
\\
a_{15} &= 362429595685359227454501841137256200262515338447122139
\\
a_{16} &= 5342698277307014122229197133594085697739662949136507986203
\\
a_{17} &= 99765301533262256100578502016534676122077769923441605548888705
\\
a_{18} &= 103139135210996186397045798509998018431340913521815632904023932244423
\\
a_{19} &= 114042545179030657632936839533863319321123228769135395651447724677783261
\\
a_{20} &= 3726987986695921904732430600737186670799479170839193448222924045573242609263
\end{aligned}
\end{equation}
\end{tiny}

%% 
%% 1,
%% 17 / asymD(1),
%% 2305 / asymD(2),
%% 4494181 / asymD(3),
%% 3330710213 / asymD(4),
%% 5712350244311 / asymD(5),
%% 52439486675194979 / asymD(6),
%% 19266759263233318405 / asymD(7),
%% 66121441024491501701765 / asymD(8),
%% 16057617271207914483637539331 / asymD(9),
%% 124141789617951906037615282061569 / asymD(10),
%% 990570538120722127305829578974187175 / asymD(11),
%% 40138653318545997972857202310993641324451 / asymD(12),
%% 29576935097999521111492046073898594892534975 / asymD(13),
%% 47226781739778967005629953528286582410693258585 / asymD(14),
%% 362429595685359227454501841137256200262515338447122139 / asymD(15),
%% 5342698277307014122229197133594085697739662949136507986203 / asymD(16),
%% 99765301533262256100578502016534676122077769923441605548888705 / asymD(17),
%% 103139135210996186397045798509998018431340913521815632904023932244423 / asymD(18),
%% 114042545179030657632936839533863319321123228769135395651447724677783261 / asymD(19),
%% 3726987986695921904732430600737186670799479170839193448222924045573242609263 / asymD(20)

Using \eqref{eq-BWYLR-even}, the denominator
\begin{equation}
T_{LR}(\e(-n/4))=\frac{1}{\sqrt{2}}(\sqrt{3}-(-1)^{(n-1)/2})
\end{equation}
has two possible values depending on $n\bmod4$. This explains the ``bimodal pattern''
observed in \cite{BWY:I}.

The case of the pA map $LR$ is rather special, and this is reflected in the
complexity of the computation as well as in the results. For example,
$T_{LR}(\e(1/n))$ (or $\tau_{LR,\lambda}(\e(2/n))$) can be computed in $O(n)$-steps
as opposed to $O(n^2)$-steps due to the fact that the double sum in the definition
decouples as a product of two single sums. The geometric representation is obtained
by the matching of two regular ideal tetrahedra of shapes $\z_6$ each and
$(\z_6)'=(\z_6)''=\z_6$, which happens to be a root of unity. In addition, the
invariant trace field $\BQ(\sqrt{-3})$ is quadratic, and the manifold is
amphicheiral, hence the coefficients of the asymptotic series are essentially
rational numbers.

\subsection{The case of $LLR$}
\label{sub.LLR}

In this section we discuss a more interesting example, namely $\varphi=LLR$. Here, we
found an interesting distinction between the 1-loop invariant $\tau_{LLR,\lambda}$
and the BWY invariant $T_{LLR}$. The phase of $\tau_{LLR,\lambda}$ has nice
asymptotics, whereas $T_{LLR}$ has small irregularities due to extra factor $\omega$
in the determinant calculation of Proposition~\ref{prop-BWY-trdet}. The results below
are stated with a mix of the 1-loop and BWY invariants, but the calculations are obtained from
the BWY invariant for efficiency.

\begin{comment}
Another is that the constant terms in the
asymptotic expansions were not in the invariant trace field $\BQ(\sqrt{-7})$ but
rather in the trace field, a quadratic extension of the invariant trace field. The
sum of two asymptotic series persisted, as did the shift of $n$ to $n-1/n$ in the
volume.
\end{comment}

If we calculate the 1-loop using \texttt{SnapPy} data, we need to take
\texttt{'b++LRL'} to compensate the cyclic permutation mentioned in
Subsection~\ref{sub.layered}. Then
\begin{equation}
\mb{G} = \begin{pmatrix}
2 & 0 & 0 \\
0 & 2 & 2 \\
0 & 0 & 0 \\
0 & -2 & 0 \\
0 & -1 & 0
\end{pmatrix}, \qquad
\mb{G}' = \begin{pmatrix}
0 & 1 & 1 \\
2 & 0 & 0 \\
0 & 1 & 1 \\
-2 & 0 & 2 \\
0 & 0 & 0
\end{pmatrix}, \qquad
\mb{G} = \begin{pmatrix}
0 & 0 & 0 \\
0 & 0 & 0 \\
2 & 2 & 2 \\
0 & 2 & 0 \\
1 & 0 & 1
\end{pmatrix}, \qquad
\boldsymbol{\eta} = \begin{pmatrix}
2 \\ 2 \\ 2 \\ 0 \\ 0
\end{pmatrix}.
\end{equation}
In \texttt{SnapPy}, the homological longitude for a once-punctured torus bundle is
the second to last equation. Thus,
\begin{equation}
A_\lambda=\begin{pmatrix}
0 & 1 & 1 \\
2 & 0 & 0 \\
-1 & -1 & 1
\end{pmatrix},\qquad
B_\lambda=\begin{pmatrix}
2 & 0 & 0 \\
0 & 2 & 2 \\
0 & -2 & 0
\end{pmatrix},\qquad
\nu_\lambda = \begin{pmatrix}
2 \\ 2 \\ -1
\end{pmatrix}.
\end{equation}
This agrees with Example~\ref{AB2} after adding the middle row to the bottom. Then
\begin{equation}
Q=2B_\lambda^{-1}A_\lambda=\begin{pmatrix}
0 & 1 & 1 \\
1 & 1 & -1 \\
1 & -1 & 1
\end{pmatrix},\qquad
\eta=2B_\lambda^{-1}\nu_\lambda=\begin{pmatrix}
2 \\ 1 \\ 1
\end{pmatrix},
\end{equation}
which match Lemma~\ref{lem.Pm}. A flattening is given in Subsection~\ref{sub.thm1}
with $f'=1,f=0$. The complete hyperbolic structure is given by
$z'_1=\frac{3+\sqrt{-7}}{8}$, $z'_2=z'_3=\frac{1+\sqrt{-7}}{4}$. Then using
\eqref{amth}, we have
\begin{equation}
\begin{split}
\tau_{LLR,\lambda}(\zeta)&=
\frac{\calD_{\zeta^{-1}}(\theta_1^{-1})\calD_{\zeta^{-1}}(\theta_2^{-1})
  \calD_{\zeta^{-1}}(\theta_3^{-1})}
{n^{3/2}\sqrt{-8\sqrt{-7}(\frac{-1+\sqrt{-7}}{8})^{1/n}}}\\
&\quad\cdot
\sum_{k}(-1)^{k_2+k_3}\theta_1^{-k_2-k_3}\theta_2^{-2k_1}
\frac{
  \zeta^{k_2^2+k_3^2+2k_1k_2+2k_1k_3-2k_2k_3+k_1+\frac{1}{2}(k_2+k_3)}
  }
{(\zeta\theta_1^{-1};\zeta)_{2k_1}(\zeta\theta_2^{-1};\zeta)_{2k_2}
  (\zeta\theta_3^{-1};\zeta)_{2k_3}},
\end{split}
\end{equation}
where $\theta_i=(z'_i)^{1/n}$ for $i=1,2,3$ and $k=(k_1,k_2,k_3) \in (\BZ/n\BZ)^3$.
This formula gives
\begin{equation}
\tau_{LLR,\lambda}(1)=(7+\sqrt{-7})^{-1/2}
\end{equation}
and, for example,
\begin{equation}
\tau_{LLR,\lambda}(\e(2/2001)) \approx (3.727322320 - 3.259362062\ii)\cdot 10^{183}.
\end{equation}

%% see: pari/constant_term_LLR.stavros.gp

\begin{comment}
The asymptotics of $LLR$ are expressed in terms of the trace field and the
invariant trace field of $M_\vphi$. These are number fields of degree $4$ and
$2$ respectively, and the fact one is an index 2 subfield of the other is due to the
2-torsion in $H_1(M_\vphi,\BZ)=\BZ+\BZ/2\BZ$, in accordance to~\cite[Cor.2.3]{NR}. 
\texttt{SnapPy} shows that the trace field of $LLR$ is $\BQ[\xi]$ of type $(0,2)$
and discriminant $2^3 \cdot 7^2$ where
\begin{equation}
\label{xidef}
\xi \approx -0.566 - 0.458 \ii, \qquad \xi^4 - \xi^3 + \xi + 1=0 \,.
\end{equation}
The invariant trace field is $\BQ(\sqrt{-7})$
where $\sqrt{-7} = -1 -2 \xi + 2 \xi^2 -2 \xi^3$ is the square root with
positive imaginary part.
\end{comment}

The complexified volume of the mapping torus of $LLR$ is given by 
\begin{equation}
\label{VCLLR}
\begin{aligned}  
V_\cx & = \text{CS}_{LLR} + \ii \text{Vol}_{LLR} \\ & = R(z_1) + 2 R(z_2)
- \frac{\pi \ii}{2} \log(z_1) - \pi \ii \log (z_2) - \frac{3}{4}\pi^2
\approx \frac{1}{8}\pi^2 + 2.66674 \ii
\end{aligned}
\end{equation}
where $R$ is the Rogers dilogarithm
\begin{equation}
\label{rogers}
R(z) = \Li_2(z) + \frac{1}{2} \log (z) \log(1-z)
\end{equation}
and $z_1=\frac{-1+ \sqrt{-7}}{2}$, $z_2=\frac{1+ \sqrt{-7}}{2}$. 
%% see: pari/precomputed_LLR.stavros.gp

The asymptotics of the 1-loop invariant we found is (subscript $LLR$ omitted for brevity)
\begin{equation}
\frac{\tau(\e(2/n))}{\delta_n \tau(1)T(\e(-n/4))}
\sim
%\alpha(1-(-1)^{\frac{n-1}{2}}\beta)
e^{\frac{v}{2}\left(n-\frac{1}{n}\right)}
\Phi\left(\frac{4\pi\ii}{n}\right),
\quad
\Phi(h) = \tau(1)\sum_{k=0}^\infty \frac{a_k}{D_k} \Big(\frac{h}{8\cdot 7 \sqrt{-7}}\Big)^k,
\end{equation}
where $v=V_\cx/(2 \pi \ii)$, 
\begin{comment}
\begin{aligned}
\alpha &=\frac{1}{\sqrt{3/2 - \xi^2 + 5/2 \xi^3}} \approx 0.6262 + 0.2097\ii,
\\
\beta &= \frac{1}{\sqrt{-1 - \xi + 2 \xi^2 - \xi^3}} \approx 0.4588 -0.5661 \ii 
\end{aligned}
\end{comment}
%such that $28\alpha^4$ is a root of
%$x^4+x^3+25x^2-48x+64$, $\beta\approx 0.458821-0.566121\ii$ is a root of
%$x^8+x^6+4x^4+x^2+1$, $v=\frac{V_{LLR}}{2\pi} - \frac{5}{16}\pi\ii$,
$\delta_n$ is a correction factor depending only on $n\bmod4$ given by
\begin{equation}
\delta_n^8=\frac{31-3\sqrt{-7}}{32},\qquad
\delta_1\approx-0.9995 + 0.0313\ii,\quad
\delta_3=e^{\pi\ii/4}\delta_1,
\end{equation}
$D_k$ is the same as in~\eqref{Dvalues}, the first few coefficients $a_k$ are given by
\begin{equation}
\label{LLRvalues}
\begin{aligned}
a_0 &= 1,\\
a_1 &= 358 - 3 \sqrt{-7},\\
a_2 &= 7 (57139 + 38532 \sqrt{-7}),\\
a_3 &= 7 (-305708866 + 1580760315 \sqrt{-7}),\\
a_4 &= 7 (-34948754616757 + 14590762181832 \sqrt{-7}),\\
a_5 &= 7^2 (-216015621732985790 + 11755310969723331 \sqrt{-7}),\\
a_6 &= 7^2 (-29690496501427874810761 - 6821015832364773754980 \sqrt{-7}),\\
a_7 &= 7^2 (-75483635753024499870522214 - 79297563089176553769763227 \sqrt{-7}),
\end{aligned}
\end{equation}
and
\begin{equation}
\begin{aligned}
T(\e(-\tfrac{1}{4}))^4&=(-24 + 18 \ii) + (-8 - 10 \ii) \sqrt{-7},&
T(\e(-\tfrac{1}{4}))&\approx -0.3194 - 1.3784\ii, \\
T(\e(-\tfrac{3}{4}))^4&=(-24 - 18 \ii) + (-8 + 10 \ii) \sqrt{-7},&
T(\e(-\tfrac{3}{4}))&\approx -2.3002 + 1.6435\ii.
\end{aligned}
\end{equation}
These values were computed using the numerically computed data at
$n=2001,\dotsc,2059$ with precision (only) 200 digits. Here, the denominator still
uses the BWY invariant $T_{LLR}$ since we lack a definition of 1-loop, and we pay the
price of an extra factor $\delta_n$.

% for the numerical computation and asymptotics, see:
% pari/BWY_LLR.gp, precomputed_LLR.gp

We believe that the shape of the asymptotics of $LLR$ persists to all pA
homeomorphisms of punctured surfaces. 

%%%%%%%%%%%%%%%%%%%%%%%%%%%%%%%%%%%%%%%%%%%%%%%%%%%%%%%%%%%%%%%%%%%%%%%%%%%%
%%%%%%%%%%%%%%%%%%%%%%%%%%%%%%%%%%%%%%%%%%%%%%%%%%%%%%%%%%%%%%%%%%%%%%%%%%%%

\section{Fourier transform and descendants}
\label{sec.FT}

In this last section we discuss the conjectural relation between the descendant
BWY invariants and the 1-loop invariants with respect to the meridian, given simply
by a Fourier transform. Note that choice of the meridian in the 1-loop invariants
was dictated by the asymptotics of the Kashaev invariant of a knot to all orders
in perturbation theory~\cite{DG2,GZ:kashaev}.

\subsection{A remark about Fourier transform}
% \label{sub.fourier}

We need to explain what it means to sum invariants that are only well-defined up to
roots of unity in Conjecture~\ref{conj.3}. The ideal answer is that there are
definitions of the invariants that do not have any ambiguities. Currently, such
definitions are not easily available, so we give a more practical explanation.
% Here, we focus on the 1-loop invariants and the descendants.
In the form \eqref{eq.ff2}, the ambiguity only comes from the choice of the $n$-th
root of $\det H$ in the definition of $T_{\varphi,\ell}$. By
Proposition~\ref{prop-BWY-trdet}, $\det H$ is actually independent of $\ell$, so it
can be factored out, making the sum well-defined.

\begin{comment}
Using the arguments
of \cite[Section~3]{DG2}, the change
$\sigma_j:\theta_i\mapsto\zeta^{-\delta_{ij}}\theta_i$ results in a factor
$\zeta^{\ell(dB^{-1})_{jN}}$ for $\tau_{M,\lambda,\ell}$ in addition to another phase
independent of $\ell$. This means after the Fourier transform, the descendant index
$m$ on the right-hand side is shifted by $(dB^{-1})_{jN}$.
\end{comment}

\subsection{Meridian for once-punctured torus bundles}
\label{sub-meridian}

Previously we ignored the sign of the homeomorphism $\varphi$ because it only
affects the meridian. However, now that we need the meridian, we will bring the sign
back into the discussion.

For once-punctured torus bundle, the layered triangulation has a canonical meridian
if the sign is $+$. This is given by the curve in the layered cusp diagram (as in
Figure~\ref{fig-cusp-layer}) connecting the centers of the triangles with the same
label, say 0. This allows us to write down the meridian equation
\begin{equation}
e^{0\pi\ii} = \prod_{i=0}^{N}
\begin{cases}
z''_{i-1},&\varphi_i=L,\\
z_{i-1}^{-1},&\varphi_i=R.
\end{cases}
\end{equation}

If the sign of $\varphi$ is $-$, the identification of the tetrahedron $T_1=T_N$ has
an extra rotation by $\pi$ compared to the $+$ case. Thus, in the layered cusp
diagram, the label $0$ in $T_N$ is identified with the label $2$ of $T_1$. To obtain
a closed curve, we need to go around once more. This gives a curve that intersects the
longitude twice, and its gluing equation is the square of the meridian equation for
$+$ as above. On the other hand, the longitude for both signs are the same. Thus, for
Conjecture~\ref{conj.3} to hold, the ``meridian" for the $-$ case needs to be half of
this curve.

A difficulty here is that with our triangulation, the matrix $B$ is always
degenerate for the meridian. It is easy to see from the meridian equation above that
the $B$ part of the meridian is all $-1$, while the sum of the rows of $B$
corresponding to $L$'s is all $2$. Thus, we cannot find a simple proof of
Conjecture~\ref{conj.3} for once-punctured torus bundles.

\begin{example}
\label{ex.LRm}
For $4_1$, the knot meridian and the mapping torus meridian agree. The descendant
version of \eqref{tau41M} is
\begin{equation}
\tau_{4_1,\mu,m}(\z) = \frac{1}{n\sqrt[4]{3}} \calD_{\z^{-1}} (\th^{-1})^2
\sum_{k, \ell \bmod n} \frac{\z^{-k \ell+m(k-\ell)} \th^{k+\ell}}{
  (\z\th^{-1};\z)_k(\z\th^{-1};\z)_\ell}.
\end{equation}
The descendant version of \eqref{tau41L} is
\begin{equation}
\tau_{4_1,\lambda,m}(\z) =
\frac{\calD_{\z^{-1}} (\th^{-1})^2}{n\sqrt{3}\zeta_6^{\frac{1-n}{2n}}} 
s_ms_{-m} \quad \text{where}\quad
s_m=\sum_{k \bmod n} (-1)^k  \frac{\z^{k^2+k/2+mk}
\th^{-k}}{(\z\th^{-1};\z)_{2k}}\,.
\end{equation}
The descendant version of \eqref{BWYLR} is
\begin{equation}
\label{TLRm}  
T_{LR,m}(q) = \frac{1}{n} \zeta_6^{\frac{n-1}{2n}} \calD_{q^{-2}}(\th^{-1})^2
\sigma_m\sigma_{-m} \,,
\end{equation}
where
\begin{equation}
\label{sm}
\sigma_m =
\sum_{k \bmod n} q^{(k^2-k)/2+mk} (-\th)^{k/2} (\th^{-1};\z^{-2})_k \,.
\end{equation}
% The relation between $\tau_{4_1,\lambda,m}$ and $T_{LR,m}$ is exactly the same as the original $m=0$ case, as in Subsection~\ref{sub.BWY41}.
\end{example}

We have checked Conjecture~\ref{conj.3} numerically for 
\begin{enumerate}
  \item $\varphi=LR$ for all odd $n\le13$,
  \item all $\varphi$ with length at most 4 for all odd $n\le9$, and
  \item a few more time-consuming examples such as $\varphi=LR$ with $\zeta=\e(1/51)$ and $\varphi=L^3R^2$ with $\zeta=\e(2/9)$.
\end{enumerate}

% see new version of one_loop_BWY_trace.sage

\subsection{$q$-holonomic aspects}
\label{sub.qholo}

Using~\eqref{sm}, one can show with an elementary computation that
$\Sigma_m=\theta^m\sigma_{2m}$ satisfies the linear $q$-difference equation
\begin{equation}
\label{rec.TDnew}  
q\Sigma_{m+1}+(q^{-4m}-q-q^{-1})\Sigma_m+q^{-1}\Sigma_{m-1}=0.
\end{equation}
% or using the \texttt{HolonomicFunctions} package of Koutschan's implementation~\cite{Koutschan} of Zeilberger's algorithm~\cite{WZ} 
Then Equation~\eqref{TLRm} implies that $T_{LR,2m}(\z)$ satisfies, as a function of
$m$, a fourth order linear $q$-difference equation that can be computed by the
\texttt{HolonomicFunctions} method~\cite{Koutschan}

%% see pari file: BWY.4.stavros.gp
%% see Mathematica file: yu-tao/RecursionBWY41.nb
\begin{tiny}
\begin{multline}
\label{TLRrec}
q^{8 m+12} \left(q^{2 m+5}-1\right) \left(q^{2 m+5}+1\right) \left(q^{4 m+10}+1\right)
\left(-q^{4 m+7}-q^{4 m+9}-q^{4 m+11}-q^{4 m+13}+q^{8 m+20}+1\right) T_{m} \\
+
q^{4 m+7} \left(q^{4 m+3}+3 q^{4 m+5}+2 q^{4 m+7}+2 q^{4 m+9}+2 q^{4 m+11}
  +2 q^{4 m+13}+q^{4 m+15}-q^{8 m+8}-2 q^{8 m+10}-3 q^{8 m+12} \right. \\ \left.
  -4 q^{8 m+14}
  -5 q^{8 m+16}-4 q^{8 m+18}-2 q^{8 m+20}-q^{8 m+22}+q^{12 m+15}+q^{12 m+17}
  +2 q^{12 m+19}+2 q^{12 m+21}+q^{12 m+23}  \right. \\ \left.
  -q^{12 m+27}-2 q^{12 m+29}-2 q^{12 m+31}
  -q^{12 m+33}-q^{12 m+35}+q^{16 m+28}+2 q^{16 m+30}+4 q^{16 m+32}+5 q^{16 m+34}
  +4 q^{16 m+36} \right. \\ \left.
  +3 q^{16 m+38}+2 q^{16 m+40}+q^{16 m+42}-q^{20 m+35}-2 q^{20 m+37}
  -2 q^{20 m+39}-2 q^{20 m+41}-2 q^{20 m+43}-3 q^{20 m+45}-q^{20 m+47}
  \right. \\ \left.
  +q^{24 m+48} +q^{24 m+50}-q^2-1\right) T_{m+1}
+
\left(q^{m+2}-1\right) \left(q^{m+2}+1\right) \left(q^{2 m+4}+1\right)
\left(q^{4 m+8}+1\right) \left(-q^{4 m+3}-q^{4 m+5}-2 q^{4 m+7} \right. \\ \left.
  -2 q^{4 m+9}
  -q^{4 m+11}-q^{4 m+13}+2 q^{8 m+10}+3 q^{8 m+12}+4 q^{8 m+14}+5 q^{8 m+16}
  +4 q^{8 m+18}+3 q^{8 m+20}+2 q^{8 m+22}-q^{12 m+17} \right. \\ \left.
  -3 q^{12 m+19}-5 q^{12 m+21}
  -7 q^{12 m+23}-7 q^{12 m+25}-5 q^{12 m+27}-3 q^{12 m+29}-q^{12 m+31}
  +2 q^{16 m+26}+3 q^{16 m+28}+4 q^{16 m+30} \right. \\ \left. +5 q^{16 m+32}+4 q^{16 m+34}
  +3 q^{16 m+36}+2 q^{16 m+38}-q^{20 m+35}-q^{20 m+37}-2 q^{20 m+39}-2 q^{20 m+41}
  -q^{20 m+43}-q^{20 m+45} \right. \\ \left. +q^{24 m+48}+1\right) T_{m+2}
+
q^{4 m+7} \left(q^{4 m+3}+2 q^{4 m+5}+2 q^{4 m+7}+2 q^{4 m+9}+2 q^{4 m+11}
  +3 q^{4 m+13}+q^{4 m+15}-q^{8 m+12}-2 q^{8 m+14} \right. \\ \left.
  -4 q^{8 m+16}-5 q^{8 m+18}
  -4 q^{8 m+20}-3 q^{8 m+22}-2 q^{8 m+24}-q^{8 m+26}-q^{12 m+15}-q^{12 m+17}
  -2 q^{12 m+19}-2 q^{12 m+21}-q^{12 m+23} \right. \\ \left.
  +q^{12 m+27}+2 q^{12 m+29}+2 q^{12 m+31}
  +q^{12 m+33}+q^{12 m+35}+q^{16 m+24}+2 q^{16 m+26}+3 q^{16 m+28}+4 q^{16 m+30}
  +5 q^{16 m+32} \right. \\ \left.
  +4 q^{16 m+34}+2 q^{16 m+36}+q^{16 m+38}-q^{20 m+35}-3 q^{20 m+37}
  -2 q^{20 m+39}-2 q^{20 m+41}-2 q^{20 m+43}-2 q^{20 m+45}-q^{20 m+47}
  \right. \\ \left. +q^{24 m+48}
  +q^{24 m+50}-q^2-1\right) T_{m+3}
+
q^{8 m+20} \left(q^{2 m+3}-1\right) \left(q^{2 m+3}+1\right) \left(q^{4 m+6}+1\right)
\left(-q^{4 m+3}-q^{4 m+5}-q^{4 m+7}-q^{4 m+9} \right. \\ \left.
  +q^{8 m+12}+1\right) T_{m+4} = 0 \,.
\end{multline}
\end{tiny}

By substituting the WKB ansatz
\begin{equation}
\label{WKB}
\tilde{\Phi}_{LR,2m}(h) = \sum_{\ell =0}^\infty c_\ell(m) \left(\frac{h}{2}\right)^j,\qquad
q=e^{h/2}
\end{equation}
in Equation~\eqref{TLRrec} where $c_\ell(m) \in \BQ(\sqrt{-3})[m]$ are polynomials
in $m$ of degree $2\ell$, we find
\begin{equation}
\label{clm}
c_\ell(m)=\sum_{k=0}^{\floor{\frac{\ell}{2}}}\tilde{a}_{\ell-2k}f_k(m)
+\sum_{k=0}^{\floor{\frac{\ell-1}{2}}}\tilde{b}_{\ell-2k}g_k(m)
\end{equation}
where $D_k$ is as in~\eqref{Dn},
$\tilde{a}_k=\left(\frac{2}{3\sqrt{-3}}\right)^k\frac{a_k}{D_k}$
is a renormalization of $a_k$ from \eqref{LRvalues}, $\tilde{b}_k$ is a new
coefficient to be determined, and $f_k(m),g_k(m)\in\BQ[m]$. The first few
values of $f_k(m)$ and of $g_k(m)$ are
\begin{tiny}
\begin{equation}
\label{fm}  
\begin{aligned}
f_0 &= 1, \\
f_1 &= -\frac{8}{3}m^4, \\
f_2 &= \frac{32}{27} m^8 - \frac{640}{81} m^6 + \frac{400}{27} m^4, \\
f_3 &= -\frac{256}{1215}m^{12} + \frac{7168}{1215}m^{10}
- \frac{180608}{3645}m^8 + \frac{1998016}{10935}m^6 - \frac{1160836}{3645}m^4
\\
%\end{aligned}
%\end{equation}
%and
%\begin{equation}
%\label{gm}  
%\begin{aligned}
g_0 &= m^2, \\
g_1 &= -\frac{8}{9}m^6 + \frac{8}{3}m^4, \\
g_2 &= \frac{32}{135}m^{10} - \frac{320}{81}m^8 + \frac{20538}{1215}m^6
- \frac{2428}{81}m^4,\\
g_3 &= -\frac{256}{8505}m^{14} + \frac{1792}{1215}m^{12} - \frac{16256}{729}m^{10}
+ \frac{1700576}{10935}m^8 - \frac{3587516}{6561}m^6 + \frac{10358761}{10935}m^4.
\end{aligned}
\end{equation}
\end{tiny}
The sequence $\tilde{b}_k$ can be determined using one descendant asymptotics
(e.g. $m=1$). With normalization
$\tilde{b}_k=-6\left(\frac{2}{3\sqrt{-3}}\right)^k\frac{b_k}{D_{k-1}}$, 
the first few values of $b_k$ are
\begin{equation}
\label{bm}
\begin{aligned}
b_1 &= 1, \\
b_2 &= 65, \\
b_3 &= 17473, \\
b_4 &= 49107541,\\
b_5 &= 48516825797, \\
b_6 &= 104606934115751, \\
b_7 &= 1158568450813142819 \,.
\end{aligned}
\end{equation}
Then the results can be checked against further descendants. We have calculated up to
$m=4$, and all terms agree.

\begin{comment}
\begin{small}
\begin{equation}
\begin{aligned}
c_0(m) &= a_0
\\  
c_1(m) &= a_1 + m^2 b_1
\\
c_2(m) &= a_2 - \frac{8}{3} m^4 + m^2 b_2
\\
c_3(m) &= a_3 - \frac{8}{3} m^4 a_1
+ \left(- \frac{8}{9} m^6 + \frac{8}{3} m^4 \right) b_1
+ m^2 b_3
\\
c_4(m) &= a_4 - \frac{8}{3} m^4 a_2
+ \frac{32}{27} m^8 - \frac{640}{81} m^6 + \frac{400}{27} m^4
+ \left(- \frac{8}{9} m^6 + \frac{8}{3} m^4 \right) b_2 + m^2 b_4
\end{aligned}
\end{equation}
\end{small}
where $a_\ell$ are given in Equation~\eqref{LRvalues} and \red{type!}

\be
\label{bl}
\begin{aligned}
b_1 &= \\
b_2 &= \\
b_3 &= \\
b_4 &= 
\end{aligned}
\ee
\end{comment}

%% see RecursionBWY41.nb, Section: WKB up to O[h]^(Lmax+1).

\subsection{The Baseilhac--Benedetti invariants}
\label{sub.BB}

The BB invariants for the $4_1$ knot are given in~\cite[Eqn.(75),p.2053]{BB:analytic}.
It is a double sum which decouples as the product of two single sums, like the
BWY invariant. With additional effort, one can try to match the sum of the
BWY invariant with that of the BB invariant.
   
\begin{conjecture}
\label{conj.4}
The invariants $\tau_{M,\lambda}(e^{2\pi i/n})/\tau_{M,\lambda}(1)$ for odd $n$ agree
with the Baseilhac--Benedetti invariants of a cusped
hyperbolic 3-manifold $M$ and its geometric representation at roots of unity.
\end{conjecture}

\subsection*{Acknowledgements} 

The authors wish to thank Thang L\^e and Campbell Wheeler for many
enlightening conversations. 

%%%%%%%%%%%%%%%%%%%%%%%%%%%%%%%%%%%%%%%%%%%%%%%%%%%%%%%%%%%%%%%%%%%%%%%%%%%%
%%%%%%%%%%%%%%%%%%%%%%%%%%%%%%%%%%%%%%%%%%%%%%%%%%%%%%%%%%%%%%%%%%%%%%%%%%%%

\bibliographystyle{hamsalpha}
\bibliography{biblio}
\end{document}